\documentclass[12pt]{article}

\usepackage[dvipdfmx]{graphicx,color}
\usepackage{amsmath,amsfonts,amssymb}
\usepackage{hyperref}
\usepackage{ulem}
\usepackage{xcolor}

\xdefinecolor{darkgreen}{rgb}{0.06, 0.56, 0.2}

\newcommand{\ol}{\overline}

\newcommand{\sr}[1]{{\cal #1}}
\newcommand{\dd}[1]{\mathbb{#1}}

\newcommand{\eq}[1]{(\ref{eq:#1})}
\newcommand{\eqn}[1]{(\ref{eqn:#1})}
\newcommand{\lem}[1]{Lemma~\ref{lem:#1}}

\newcommand{\pro}[1]{Proposition~\ref{pro:#1}}

\newcommand{\rem}[1]{Remark~\ref{rem:#1}}

\newcommand{\app}[1]{Appendix~\ref{app:#1}}
\newcommand{\sectn}[1]{Section~\ref{sec:#1}}

\newcommand{\pend}{\hfill \thicklines \framebox(6.6,6.6)[l]{}}

\newenvironment{proof}{\noindent {\sc  Proof.} \rm}{\pend}
\newenvironment{proof*}[1]{\noindent {\sc  #1} \rm}{\pend}

\newtheorem{theorem}{Theorem}[section]
\newtheorem{lemma}{Lemma}[section]

\newtheorem{proposition}{Proposition}[section]

\newtheorem{remark}{Remark}[section]

\newtheorem{example}{Example}[section]

\newtheorem{corollary}{Corollary}[section]

\newcommand{\setsection}[2] {
\setcounter{section}{#1}
\setcounter{subsection}{0}
\setcounter{equation}{0}
\setcounter{conjecture}{0}
\setcounter{assumption}{0}
\setcounter{question}{0}
\setcounter{definition}{0}
\setcounter{theorem}{0}
\setcounter{corollary}{0}
\setcounter{lemma}{0}
\setcounter{proposition}{0}
\setcounter{remark}{0}
\setcounter{appen}{0}
\setsection*{\large \bf \thesection. #2}}

\newcommand{\setnewcounter} {
\setcounter{subsection}{0}
\setcounter{equation}{0}
\setcounter{conjecture}{0}
\setcounter{assumption}{0}
\setcounter{question}{0}
\setcounter{definition}{0}
\setcounter{theorem}{0}
\setcounter{corollary}{0}
\setcounter{lemma}{0}
\setcounter{proposition}{0}
\setcounter{remark}{0}
}

\begin{document}
\title{\bf \Large Heavy-tail asymptotics for the length of a busy period in a Generalised Jackson Network}

\author{Sergey Foss \footnote{Email: \href{mailto: S.Foss@hw.ac.uk}{s.foss@hw.ac.uk}}\\Heriot-Watt University and \\  Sobolev Institute of Mathematics
\\ \and Masakiyo Miyazawa \footnote{Email: \href{miyazawa@rs.tus.ac.jp}{miyazawa@rs.tus.ac.jp}}\\ Tokyo University of Science\\ \and Linglong Yuan \footnote{Email: linglong.yuan@liverpool.ac.uk. Research of L.Y. is supported by
 the Heilbronn Institute for Mathematical Research (HIMR) and the UKRI/EPSRC
Additional Funding Programme for Mathematical Sciences.} \\
University of Liverpool}
\date{\today}

\maketitle
\begin{abstract}

We consider a Generalised Jackson Network  with finitely many servers, a renewal\ input and $i.i.d.$\ service times at each queue. We assume the network to be stable and, in addition, the distribution of the inter-arrival times to have unbounded support. This implies that the length of a typical busy period $B$, which is the time between two successive idle periods, is finite a.s. and has a finite mean.

We assume that the distributions of the service times with the heaviest tails belong to the class of so-called intermediate regularly varying distributions. 
We obtain the exact asymptotics for the probability ${\dd P} (B>x)$, as $x\to\infty$. For that, we show that the Principle of a Single Big Jump holds: $B$ takes a large value mainly due to a single unusually large service time. 
\end{abstract}
  
\begin{quotation}
\noindent {\bf Keywords:} Generalised Jackson Network, Busy Period, Tail Asymptotics, (Intermediate) Regularly Varying Distribution, Principle of a Single Big Jump, 
Fluid Network

\vspace {1mm}
\noindent\textit{MSC (2020): } 60K25, 60F10, 60K20
\end{quotation}
\newpage
\tableofcontents

\section{Introduction}
\label{sec:introduction}

Uma Prabhu was one of the founders of the modern queueing theory and the founder of the ``Queueing systems: theory and applications'' journal. He made significant contributions to the performance and asymptotic analysis of many important characteristics of queueing systems and, in particular, of the length of a busy period, see e.g. Sections 1.1.4, 1.3.3, 3.6.5 and 3.7.3 in his monograph \cite{Prab}.

We consider tail asymptotics problems of certain characteristics of queueing networks. 
In the classical light-tail scenario (where the distributions of service times possess finite exponential moments), 
the tail asymptotics have been studied in queueing networks 
and related reflecting processes for many years
(see e.g. \cite{DaiMiya2011,Miya2011} and references therein). 
In the case of heavy-tailed distributions of service times, substantial results have been obtained mostly for random processes on the real line, including queueing systems, regenerative processes, Levy processes etc. 
Namely, tail asymptotics have been found for the distributions of the waiting or sojourn 
in single- and multi-server queues and then for feedforward networks of such queues (see e.g. \cite{FKZ, ZwarBorsMand2004, LiesMand2008, Kors2018} and the lists of references therein)  
and also for the so-called maximal dater in the Generalised Jackson Networks (GJNs, in short) (see e.g.\ \cite{Lela2005a}, \cite{BaccFossLela2005} ). 
Further, in the stable single-server queue, the distributional tail asymptotics have been found for the length of a typical busy period, but only in two cases, 
where either the service time distribution is intermediate regularly varying (see e.g. \cite{Zwar2001, FossMiya2018, AsmuFoss2018}) or the so-called ``square-root insensitive'' (see e.g. \cite{jelenkovic2004large, JeleMomcZwar2004})  -- in the terminology of \cite{jelenkovic2004large}, but we are unaware of any results in the case where, say, the tail distribution of the service time is not square-root insensitive  and, for example, for $\overline{F}(x) = \exp (-x^{\beta})$ with $\beta \in (1/2, 1)$.
The main difficulty in analysing the busy period is that its length may decrease with an increase of inter-arrival  times, which ruins the corresponding monotonicity property.
The complexity of the tail asymptotics in the ``non-square-root insensitivity" domain has been shown in the papers \cite{AsmuKlupSigm1998} and \cite{FossKors2000}.

We mention also 
a number of papers on feedforward networks with heavy-tailed service time distributions that include fluid queues with jump inputs, L\'{e}vy-driven queues, parallel queues, coupled queues etc. (see, e.g., \cite{LiesMand2008,ZwarBorsMand2004} and references therein). 
In \cite{FossMiya2018}, the tail asymptotics for the sojourn time of a ``typical'' customer with a single-server queue with feedback has been analysed.

In this paper, we consider a more general stochastic queueing model where the characteristic of interest,
the length of a busy period, is ``non-monotonic'' on the inter-arrival times of customers. Namely, we deal with the so-called (open) Generalised Jackson Network   with a finite number of single-server stations. In what follows, it is convenient to us to identify the terms ``server $k$'' and ``station $k$'' and use both of them, whichever is more appropriate in a particular case. We assume that 
the exogenous input forms a renewal process where each arriving customer is directed randomly (according to predefined probabilities) to a server, and that the service times at each server are independent and identically distributed (but their distributions may differ for different queues). In front of each server, customers form a queue and are served in the order of their arrivals. After service completion, the served customer either randomly chooses one of the servers for the next service or leaves the network. It is known that if at each queue  the total arrival 
rate of customers is lower than the service rate, then the network is  ``stable'' in the sense that the distributions of its queue lengths and of corresponding workloads converge (either weakly or in the total variation norm -- depending on extra technical assumptions) to the unique limiting/stationary distribution. If, in addition, the distribution of the inter-arrival times has unbounded support, the whole network empties regularly almost surely, so one can talk about a ``typical'' busy period, which is the time between two consecutive idle periods. We are interested in the analysis of the
length of the first busy period (call it $B$) and, namely, in finding the asymptotic expression for ${\dd P} (B>x)$, as $x$ grows to infinity. With a slight abuse of notation, the period is also referred to as the period length; so we use $B$ for the first busy period, and also for its length, depending on the context.

We assume that the service time distributions with the ``heaviest'' tails belong to the class ${{\sr{IRV}}}$ of so-called intermediate regularly varying distributions. Recall that a distribution function $F\in {{\sr{IRV}}}$ if
\begin{align}\label{eq:IRV0}
\lim_{y\downarrow 1}
\liminf_{x\to\infty} \frac{\overline{F}(yx)}{\overline{F}(x)}=1
\end{align}
where $\overline{F}(x) = 1-F(x)$ is the tail distribution function. 
It is known (but, probably, not well-known) that a distribution function $F$ 
is ${{\sr{IRV}}}$ if and only if
\begin{align}\label{IRV-1}
\frac{\overline{F}(x+h(x))}{\overline{F}(x)} \to 1 \ \ \text{as} \ \ x\to\infty,
\end{align}
for any function $h(x)$ which is of order $o(x)$ as $x\to\infty$
(this means that $\lim_{x\to\infty} h(x)/x = 0$),
and if and only if
\begin{align}\label{IRV-2}
\frac{\overline{F}(S_n)}{\overline{F}(n\mathbb EX_1)
} \to 1 \ \ \text{a.s. as} \ \ n\to\infty, 
\end{align}
for the partial sums $S_n=\sum_{i=1}^nX_i$ of {\it any} $i.i.d.$\ sequence $\{X_n\}$ of random variables  with a strictly positive finite mean, see Theorems 2.47 and 2.48 in \cite{FKZ}. 
Many results in this paper are based on these characteristic properties \eqref{IRV-1} and \eqref{IRV-2}.

The paper is organised as follows. 
In Section \ref{sec:main}, we introduce the model  
and formulate our main result Theorem \ref{th1}. Then in Section \ref{sec:auxiliary}  
{we present an alternative coupling representation of the network}, 
formulate a number of monotonicity and invariance properties, introduce several auxiliary models,  and present the scheme of the proof of the main result Theorem \ref{th1}.
In Section \ref{sec:ulbound}, we propose upper and lower bounds for the length of the first busy period. Then, in Section \ref{sec:PSBJ}, we establish the principle of a single big jump, first for the upper bound and then for the busy period length $B$. Then we obtain the desired asymptotics in Section \ref{sec:finduk} and illustrate the result by an example of a two-station network.
Finally, Appendices help to keep the paper self-contained:
{in Appendix A, we formulate basic properties of some classes of heavy-tailed distributions; in Appendix B, we formulate a result on the tail asymptotics for random sums; Appendix C includes the tail-asymptotics result for the busy period in a single-server queue; Appendix D provides some simple bounds on the times to empty the GIN; and Appendix E summarises known results on fluid limits in GJNs.   }

{\bf Notation.} In what follows, we say that two (strictly) positive functions $f(x)$ and $g(x)$ are {\it asymptotically equivalent}
(at infinity) and write 
\begin{align}\label{equi}
f(x)\sim g(x) 
\end{align}  
if $\lim_{x\to\infty} f(x)/g(x) =1$.
Next, $f(x)$ {\it asymptotically dominates $g(x)$}, 
\begin{align*}
f(x) \gtrsim g(x)
\end{align*}
if $\liminf_{x\to\infty} f(x)/g(x)\geq 1.$ Further, recall the standard notation: 
\begin{align*}
f(x) = O(g(x))
\end{align*}
if $\limsup_{x\to\infty} f(x)/g(x) < \infty$, and
\begin{align*}
f(x) = o(g(x))
\end{align*}
if $\lim_{x\to\infty} f(x)/g(x) = 0.$

We say that two families of events, $\{ A(x) \}$ and $\{B(x)\}$, of positive probabilities are asymptotically equivalent and write 
\begin{align}\label{setequi}
A(x) \simeq B(x)
\end{align} 
if ${\mathbb P} (A(x)\setminus B(x)) + {\mathbb P} (B(x) \setminus A(x)) = o({\mathbb P} (A(x))$ as $x\to\infty$. 
Clearly, \eqref{setequi} implies that ${\mathbb P} (A(x)) \sim {\mathbb P} (B(x))$ in the sense of \eqref{equi}, but not the other way around. 

We write  
\begin{equation}
A(x)\stackrel{\sim}{\subset}B(x)
\end{equation}
if $\mathbb P(A(x)\setminus B(x))=o(\mathbb P(A(x))$.

We use the same notation, say $F$, for the distribution of a random variable and for its distribution function, and then $\overline{F}(x) =1 - F(x)$ is the tail distribution function.

For two positive random variables $\alpha(x)$ and $ \beta(x)$, we use notation $\alpha(x)\approx_p \beta(x)$ if there exist two positive constants $0<c<C$ such that 
\begin{align*}
\mathbb P(c\beta(x)\leq \alpha(x)\leq C\beta(x))\xrightarrow[x\to\infty]{}1,
\end{align*}
and we write $\alpha(x) \sim_p \beta(x)$ if 
$\alpha(x)/\beta(x) \to 1$ in probability, as $x\to\infty$.

We write $\alpha(x)=O_p(\beta(x))$ if $\lim_{x\to\infty}\mathbb P(\alpha(x)\leq C\beta(x))=1$ for some $C>0,$ and write $\alpha(x)=o_p(\beta(x))$ if $\alpha(x)/\beta(x)$ converges in probability to $0$ as $x\to\infty.$

Assume that an event $A=A(x)$ depends on parameter $x\ge 0$. 
We say that $A$ occurs {\it with high probability}, or {\it w.h.p.} in short, if ${\mathbb P}(A) \to 1$ as $x\to\infty$. 

We write for short $[n]:=\{1,2,\ldots,n\}$, and this notation is typically used for $n=K$ and $n=K+1$.

\section{Description of the model and the main result}
\label{sec:main}
\setnewcounter

In this section, we introduce a Generalised Jackson Network, formulate our main result Theorem \ref{th1}. 
Recall that 
the meaning of ``generalised'' is that both inter-arrival and service times are not required to have exponential distributions.

We consider an open queueing network with $K$ single-server stations.   Customers $n=1,2,\ldots$ arrive at time instants $0\le T_1\le T_2 \le T_n \le \ldots$. Setting $T_0=0$, we assume that inter-arrival times $t_n=T_{n}-T_{n-1}, n\ge 1$ form an $i.i.d.$\ sequence with finite positive mean $a={\mathbb E} t_1$ and that the first customer arrives in an empty network. Any arriving customer $n$ is directed to the server $\alpha_{n,0}\in[K]$ and joins the queue there (or takes service immediately if the server is idle). We assume that all $\alpha_{n,0}$'s are $i.i.d.$ such that $\mathbb P(\alpha_{n,0}=k)=p_{0,k}$ where $
p_{0,k}\ge 0$ and $\sum_{k=1}^K p_{0,k}=1$. 
At each station, customers are served one-by-one in the order of their arrivals {\it to that station}, and we denote by 
$\sigma_{k,i}$ the duration of the $i$'th consecutive service at station $k$ (this means that here we associate service times with stations rather than with customers). We assume that, for any $k$, random variables $\{\sigma_{k,i}\}$  are $i.i.d.$\ with finite mean $b_k$ and distribution function
$F_k$. We use $\alpha_{k,i}\in[K+1]$ to denote the next server after completing the $i$-th consecutive service at station $k$ where $\alpha_{k,i}=K+1$ means it leaves the network. We require that $\mathbb P(\alpha_{k,i}=\ell)=p_{k,\ell}$ where $p_{k,\ell}\ge 0$ and $\sum_{\ell=1}^{K+1} p_{k,\ell}=1$. We assume that all $t_n$'s, $\alpha_{n,0}$'s, $\sigma_{k,i}$'s and $\alpha_{k,i}$'s are mutually independent. 

The inter-arrival times $\{t_n\}$ and the collection of random variables
\begin{align*}
\Sigma_S:=\Big\{\{\sigma_{k,i}\}_{i\geq 1}, \ \ 1\leq k\leq K ;\quad \{\alpha_{k,i}\}_{i\geq 1}, \ \ 0\leq k\leq K\Big\},
\end{align*}
completely determine the dynamics of the network. In what follows, we call $\Sigma_S$ 
the \textit{server-associated network infrastructure.} This infrastructure provides a coupling construction, 
that will be used for various arrival times.

In the model described above, we consider the case where every customer eventually leaves the network with probability 1. 
This means that $K+1$ is the only absorbing state for  the {\it auxiliary Markov chain} $\{Y_j\}$ with initial state $Y_0=0$, that lives in the  state space
$\{0,1,2,\ldots,K,K+1\}$ and has transition probabilities $\{p_{k,\ell}\}$ (we assume in
addition that $p_{K+1,K+1}=1$). 
Let $N_k$ be the number of times the Markov chain $\{Y_j\}$  visits state $k$, i.e. $N_k = \sum_{j=1}^{\infty}
{\bf 1} (Y_j=k)$, and let $\widehat N= \sum_{k=1}^K N_k$ be their total. In order to avoid trivialities, we may assume that
$\dd{P}(N_k\ge 1)>0$, for all $k \in [K]$. By the basic theory of finite Markov chains, $\widehat N$ has a  finite exponential moment,
\begin{align}\label{eq:exp}
\dd{E} e^{c\widehat N}<\infty, \quad \mbox{for some} \quad c>0.
\end{align}

Return to the Generalised Jackson Network. Let $Q_{n,k}$ be the queue length at the station $k \in [K]$ that is observed by customer $n$ at the time $T_n$ of its arrival, and $R_{n,k}$ the residual service time of the customer in service if there is any (and we let $R_{n,k}=0$ if the server is idle). The sequence $Z_n = ((Q_{n,k}, R_{n,k}), k \in [K])$ forms a time-homogeneous Markov chain that describes the dynamics of the network. 

In the proposition below, we summarise some known stability results (see, e.g., \cite{Foss1991} or \cite{BaccFoss1994} or \cite{DownMeyn1994} or \cite{dai1995positive}). We present these results in the generality needed for our purposes.  

\begin{proposition}
\label{pro:Prop1}
Assume that the following two conditions hold: 
\begin{align}
\label{eq:stab}
a> \max_{1\le k \le K} b_k \dd{E} N_k,
\end{align}
and the distribution of the inter-arrival times has unbounded support, 
\begin{align}
\label{eq:unbound}
{\mathbb P} (t_1>x)> 0, \ \ \text{for all} \ \ x>0,
\end{align}
\begin{itemize}
\item[(i)]
Then the GJN is {\it stable}, in the sense that the underlying Markov chain $\{Z_n\}$ admits a unique stationary distribution, say $\pi$, and that the distributions of $Z_n$ converge to $\pi$ in the total variation norm, for any initial value $Z_0$.
\item[(ii)]
Further, the Markov chain possesses a positive recurrent atom ${\bf 0}$ with all coordinates equal to zero. This means that the state ${\bf 0}$ is achievable from any other state of the Markov chain and that the mean return time to state ${\bf 0}$ is finite. 
\end{itemize}
\end{proposition}

We consider the case where some of service time distributions are {\it heavy-tailed}, namely, the following assumptions are in force.
 
 {\bf Heavy-tail assumptions (HTA).} There is  a reference 
distribution function $G$ which is {\it intermediate regularly varying} (see \eq{IRV0}) 
 and such that
\begin{align}\label{eq:ma}
\dd{P}(\sigma_{k}>x) \sim c_k \overline{G}(x), \,\,x\to\infty,
\quad \mbox{where} \ c_k\ge 0 \ \mbox{and} \ c = \sum_{k=1}^K c_k > 0.
\end{align} 
Here $\sigma_k$ has the same distribution $F_k$ as all $\sigma_{k,i}$'s and the above equivalence with $c_k=0$ is interpreted as $\dd{P}(\sigma_{k}>x)=o(\overline G(x)), x\to\infty.$ 

Let $\theta_n$ be the time instant when customer $n$ leaves the network  and let 
\begin{align}\label{eq:nuB}
\nu = \inf \{ n: \max_{m\le n} \theta_m < T_{n+1} \}  
\end{align}
and 
\begin{align}\label{eq:defB}
B = \sup_{m\le \nu}\theta_{m} - T_1, \ \ \text{where} \ \ \sup_{m\le\infty}
\theta_{m} = \infty, \ \text{by convention}.
\end{align}
We call $B$ the (length of the) first busy period, and then $\nu\equiv\nu_B$ is the number of customers served within the first busy period. In general, $\nu_B$ and $B$ may take infinite values even if the stability condition \eq{stab} holds as shown in an example below.

\begin{example}
\label{exa:tandem}
Consider a simple stable tandem queue with $K=2$  stations, 
$i.i.d$. inter-arrival times $t_n, n\ge 1$ having a uniform distribution $U$ in the interval $(5/4, 7/4)$, transition probabilities $p_{0,1}=p_{1,2}=p_{2,3}=1$ and dererministic service times $\sigma_{1,j}\equiv \sigma_{2,j} \equiv 1$.   In this example, at least one customer 
is present in the network at any time and, therefore, it never empties, $\nu_B=B=\infty$ a.s.
\end{example}

However, $\nu_B$ and $B$ are $a.s.$\ 
finite and have finite expectations if we assume the condition \eq{unbound} in addition to \eq{stab},
as it follows from item (ii) of \pro{Prop1}.

Here is our main result:
\begin{theorem}\label{th1}
Under the assumptions \eq{stab}, \eq{unbound} and \eq{ma}, the following asymptotic equivalence holds: 
\begin{align}\label{eq:th1}
\dd{P}(B>x) \sim  \dd{E} \nu_B \sum_{k=1}^{K} c_k\dd{E}N_k\ol{G}(u_kx),
\end{align}
where 
 the coefficients 
$u_k$ are determined by the algorithm presented in Section \ref{sec:algo} and the terms with $c_k=0$ may be omitted.
\end{theorem}
\begin{remark}
In fact, \eqref{eq:th1} is a consequence of a slightly stronger result, which shows the principle of single big jump for $B,$ see Section \ref{Sec3.4} for more detail.
\end{remark}

 \section{Auxiliary models and properties, and the scheme of the proof of Theorem \ref{th1}}
\label{sec:auxiliary}
\setnewcounter

The description of the GJN given in the previous section is based on the {\it server-associated} infrastructure $\Sigma_S$, with a complete ordering of service times at each server and of transition links from each server. In this section, we recall another  coupling construction $\Sigma_C$, then truncate them and concatenate, to obtain the earlier random variales for the GJN which is {\it customer-associated} (see \cite{BaccFoss1994}),  which {is  used in the proof of our main theorem and helps to clarify  certain sample-path properties that are recalled here too. }Then we introduce a number of auxiliary models. We complete the section with presenting the scheme of the proof of our main result. 

    \subsection{Alternative coupling representation}
In the previous section, we introduced a Markov chain $\{Y_j\}$. 
Now we extend its representation by including service/holding  times at each station and denote the extension by $\{Y_{1,j}\}$
 (with adding an extra index $1$) such that $Y_{1,0}=(0,0), Y_{1,j}\in [K]\times (0, \infty)$ for $1\leq j< \widehat N,$ with the absorbing state  $Y_{1,\widehat N}=(K+1,0).$ The new Markov chain is characterised by a collection of random variables
\begin{align}\label{wideY}
\{\alpha_{1,0}, \{\sigma_{1,k,i}\}_{1\leq k \leq K, 1\leq i \leq N_{1,k}}, 
\{\alpha_{1,k,i}\}_{1\leq k \leq K, 1\leq i\leq   N_{1,k}} \}.
 \end{align}
Here $N_{1,k}=N_k$, $\alpha_{1,0}$ is the 
station that the Markov chain visits first, ${\mathbb P}(\alpha_{1,0}=k) = p_{0,k}$, $\sigma_{1,k,i}$ 
is the duration of the $i$'th service/holding time at station $k$ and $\alpha_{1,k,i}$ the next server the Markov chain is directed to after this service, 
${\mathbb P}(\alpha_{1,k,i}=j)=p_{k,j}$ --  or
the exit gate if $j=K+1$.

{Let $\{Y_{2,j}\}, \{Y_{3,j}\}, \ldots$ be $i.i.d.$ copies of $\{Y_{1,j}\}$. Then $\{Y_{n,j}\}$ is  characterised by 
\[\left\{\alpha_{n,0}, \{\sigma_{n, k,i}\}_{1\leq k\leq K, \,\, 1\leq i\leq N_{n,k}},   \{\alpha_{n, k,i}\}_{1\leq k\leq {K},\,\, 1\leq i \leq N_{n, k}}\right\},\quad \forall n\geq 1.\]
Together, these sequences form the \textit{customer-associated network infrastructure} $\Sigma_{C}$, of the GJN:
\begin{align*} 
\Sigma_C :=\left\{\alpha_{n,0}, \{\sigma_{n,k,i}\}_{1\leq k\leq K, \,\, 1\leq i\leq N_{n,k}},   \{\alpha_{n,k,i}\}_{1\leq k\leq {K+1},\,\, 1\leq i \leq N_{n,k}}\right\}_{n\ge 1}
\end{align*}
 The two infrastructures are equivalent. To show that,  we put all service times and transition links at any server $k$ altogether and in  accordance with the natural order:
\begin{align}\label{s-ordered}
&\sigma_{1,k,1},\ldots, \sigma_{1,k,N_{1,k}}, \sigma_{2,k,1},\ldots, \sigma_{2, k,N_{2,k}},\sigma_{3,k,1},\ldots
\end{align}
and
\begin{align}\label{a-ordered}
&\alpha_{1,k,1},\ldots, \alpha_{1,k,N_{1,k}}, \alpha_{2,k,1},\ldots, \alpha_{2, k,N_{2,k}},\alpha_{3,k,1},\ldots.
\end{align}
Due to the strong Markov property, random variables in each of the sequences (\ref{s-ordered}) and (\ref{a-ordered}) are i.i.d., and the sequences do not depend on each other. 
We rename the $\sigma$'s and the $\alpha$'s in the above sequences as 
\begin{align}\label{eq:sigmakl}\sigma_{k,1},\sigma_{k,2},\ldots \quad \text{and}
\quad
\alpha_{k,1},\alpha_{k,2},\ldots .
 \end{align}
 Thus, we obtained the server-associated infrastructure $\Sigma_S$, and one can see that there is the one-to-one correspondence between the two infrastructures: if we let
 $V_{n,k}= \sum_{i=1}^n N_{i,k}$, then 
\begin{align}\label{sigma-alpha}
\sigma_{n,k,i}= \sigma_{k,V_{n-1,k}+i}
\ \ 
\text{and} 
\ \ 
\alpha_{n,k,i}= \alpha_{k,V_{n-1,k}+i},
\end{align}
for any $n\geq 1$ and $k\in [K]$ and for  all $1\le i \le N_{n,k}$.  
 }
The reason we call it \textit{customer-associated} lies in the so-called \textbf{isolation model}, where we assume that customers do not overlap in the network and that any next customer enters the network when the previous customer leaves the network. 
Then the collection of random variables $\{Y_{n,j}\}$ provides all the information for customer $n$'s service and routing in the network.
In what follows, we will use both infrastructures $\Sigma_S$ and $\Sigma_C$, see further discussion in
Section \ref{sec:concluding}.

\subsection{Sample-path and distributional properties}

With the coupling construction introduced above, we have a number of
sample-path monotonicity and invariance properties.
Let $I_{k,\ell}$ be the time of the start of the $\ell$'th service at station $k$ and $D_{k,\ell}$ the time when it finishes. 

The following properties hold.

{\bf Monotonicity property 1 (MP1)} 
For all $\ell$ and $k$, both $I_{k,\ell}$ and $D_{k,\ell}$ are monotone non-decreasing functions of all arrival times $T$'s and service times $\sigma$'s. In addition, if one blocks/delays the beginning of some service, this leads to delays of subsequent services, too.

{\bf Invariance property (IP)} Assume that $T_0=0\le T_1 \le \ldots \le T_n <\infty$ and $T_{n+1}=\infty$, for some $n$. Then the total number of services of the $n$ customers at any station $k$ is the same for all finite values of $T_1,\ldots,T_n$ and for all finite values of service times.

{\bf Monotonicity Property 2 (MP2).} For any fixed sequence 
$\{T_n\}$,  the values of $B$ and $\nu_B$ are monotone increasing functions of the $\sigma$'s and of the lengths of delaying services (if any). 

However, the dependence of $B$ and $\nu_B$ on the $T$'s  is {\bf non-monotonic!} -- both are neither  increasing nor decreasing functions on the arrival times. 

The MP1 property was proved in \cite{Foss1989} and \cite{ShanYao1989} and the IP property in \cite{BaccFoss1994}.
The MP2 property follows directly from MP1 and from the definitions, see equations  \eq{nuB} and \eq{defB}.

There is another monotonicity property that is distributional and may be made sample-path by a proper coupling construction of two networks, see
\cite{BaccFoss1994} for more details.

{\bf Monotonicity Property 3 (MP3)}
Consider two Generalised Jackson Networks with the same number of stations, the same distributions of inter-arrival and service times and with the same matrix of transition probabilities. Assume that at the initial time $t=0$ there are already $x_k\ge 0$ (correspondingly, $x_k'$) customers at stations $k=1,2,\ldots,K$ in the first (correspondingly, in the second)  network with the first customers in queues having residual service times $R_k$ (correspondingly, $R_k'$). Let $B$ and $B'$ be the corresponding times to empty the networks and $\nu_B$ and $\nu_B'$ the numbers of customers served before that times. 
If, for any $k$, 
\begin{align*}
x_k\geq x_k' \ \ \text{and} \ \ R_k \geq R_k',
\end{align*}
then 
\begin{align*}
{\mathbb P}(B>x)\ge {\mathbb P}(B'>x) \ \ \text{and} \ \ {\mathbb P}(\nu_B>x) \ge {\mathbb P}(\nu_B'>x), \ \ \text{for all}\ x.
\end{align*}
 
\subsection{Auxiliary models}\label{sec:auximodel}
With keeping the same network infrastructures $\Sigma_S$ 
and $\Sigma_C$, 
we introduce a number of models with various input streams and delays.

Namely, we consider three input streams: the original $\{T_n\}$, the {\it $L$-constrained} input stream $\{T_n'\}$ where $T_n'=T_n$ for $n=1,2,\ldots,L$ and $T_n'=\infty$ for $n>L$, and the {\it saturated stream}, $\{T_n^0\}$ where all (infinitely many) customers arrive instantaneously at time $T_n^0=0$, $n\ge 1$.  

Next, we fix an integer $L\ge 1$ and consider a model where customers are served in groups of size $L$, as follows. 

Start with the {\bf saturated group-$L$ model}. We assume that infinitely many customers numbered $n=1,2,3,\ldots$ are waiting at time zero at the entrance gate. We introduce the entrance regulations inductively. We allow the first $L$ customers (numbered $1,2,\ldots,L$) to enter the network and to proceed with service, and hold all other customers at the gate. Then at time instant $X_{1,L}^0$ when the last of the first $L$ customers leaves the network we allow the next $L$ customers (numbered $L+1,\ldots,2L$) to enter and get serviced.
It takes these customers, say, $X_{L+1,2L}^0$ units of time to get serviced and leave. Then, at time $X_{1,L}^0+X_{L+1,2L}^0$, we allow the next $L$ customers to enter the service, and so on. 
One can see that the sequence $\{X_{(m-1)L+1, mL}^0\}$ is $i.i.d.$

The {\bf saturated group-$1$} model is of particular interest to us, and is exactly the {\bf isolation model} that we introduced earlier. Here, 
for any $n=1,2,\ldots$, the $(n+1)$'st customer enters the network and proceeds with service when the $n$'th customer leaves the network. In this model, let $\beta_{n,k,\ell}$ be the number of transitions of customer $n=1,2,\ldots$ from station $k$ to station $\ell$. In this model  $N_{n,k}$ 
is the total number of services of customer $n$ by station $k$.  Clearly, $\beta_{n,k,\ell} = \sum_{r=1}^{N_{n,k}} 
{\bf 1} (\alpha_{n,k, r}=\ell)$ and, further, $N_{n,k} = \sum_{\ell=0}^K \beta_{n,\ell,k} = \sum_{\ell=1}^{K+1} \beta_{n,k,\ell}$, which means that the number of services of customer $n$ by server $k$ is equal to the number of the customer's visits to $k$ (i.e. arrivals to/departures from). Let us repeat that, for any $n$, the random variables $\{N_{n,k}, k=1,2,\ldots,K\}$ have the same joint distribution as random variables $\{N_k, k=1,2,\ldots,K\}$, where $N_k$ is the number of visits to the station $k$ by the auxiliary Markov chain $\{Y_j\}$ introduced earlier.

{
One can see, that in the isolation model it is natural to use the $\Sigma_C$ infrastructure.
Here we have split all service times and all corresponding links into groups of random sizes, $\{\sigma_{n,k,i}\}_{i=1}^{N_{n,k}}$ and
 $\{\alpha_{n,k,i}\}_{i=1}^{N_{n,k}}$ where the two groups with index $n$ relate to customer $n$, and we may say that these services and links {\it are brought and used} by customer $n$ in the isolation model.  
}  

For any $n$, the time $X_{n,n}^0$ is the one that customer $n$ is served for in isolation in the saturated group-1 model, and it is nothing but $\widetilde{S}_n$, the total service time of customer $n$:
  \begin{align*}
  \widetilde{S}_n :=
  \sum_{k=1}^K \sum_{j=1}^{N_{n,k}} \sigma_{n,k,j}.
  \end{align*}
  
Due to the MP1 property,
\begin{align}\label{eq:finite}
X_{1,L}^0 \le \sum_1^L \widetilde{S} _i =: S_{1,L} \ \text{a.s.}
\end{align}  
{{
\begin{remark}\label{REM}
Due to the IP property, the total number of services received by the first $L$ customers from any station $k$ is the same in the group-$1$ and the group-$L$ saturated models. However, the durations of services of a certain customer and its further output links may depend on the order of customers' arrivals to the servers
\end{remark}
\begin{remark}\label{REM11}In the isolation model, 
  the service times $\{\sigma_{n,k,i}\}$ and the transition links $\{\alpha_{n,k,i}\}$
  have a natural meaning. In other models, they may be interpreted differently. Namely, for any $n$, this collection may be viewed as service times and transition links that are brought to the network by customer $n$, but are used by all customers in the order of their arrivals to the servers. Since the customers do not overlap in the isolation model, the service times and the links are used there by customer $n$ itself, but may be used by other customers in other models. 
\end{remark}

Turn to the {\bf  second auxiliary model} with arrival input $\{T_n'\}$. Here only $L$ customers enter the network in a finite time. Let $X_{1,L}$ be the duration of the activities in this network, i.e.
\begin{align}\label{defX}
X_{1,L} = \max_{n\le L} \theta_n - T_1,
\end{align}
where $\theta_n$ is the departure time for customer $n$. 

Now we introduce the {\bf third auxiliary model} with service delays. We call it the single-server {\bf upper-bound queue (UBQ)}.  

Fix again $L\ge 1$. Consider the Generalised Jackson Network with arrival times $\{T_n\}$ and delay services using again the entrance gate. 
 First, customers $1,2,\ldots,L$ are stopped at the entrance until time $T_L$ of the arrival of the $L$'th customer (therefore the UBQ starts its service at time $T_L$).  All these customers enter the network at time $I_1:=T_L$.
Then the next $L$ customers numbered $L+1,\ldots,2L$ are stopped at the entrance gate until time
$I_2:=\max (T_{2L},I_1+X_{1,L}^0)$, the first time when all these customers are present at the entrance gate and all previous customers have
left the network. Inductively, customers numbered $nL+1,\ldots,(n+1)L$ form the $(n+1)$-th
group, and they are stopped at the gate until time $I_{n+1}$ when the last of two events happens: all of them arrive to the gate and all customers from the previous group have left the network, so $I_{n+1}=\max (T_{(n+1)L}, I_n+X_{nL+1,(n+1)L}^0)$.
This model may be viewed as a single-server queue with new i.i.d.\ service times 
\begin{align*}
\widehat{\sigma}_n := X_{(n-1)L+1, nL}^0
\end{align*} 
with mean $\dd{E}X_{1,L}^0$ (which is finite, thanks to \eq{finite}), and new i.i.d.\ inter-arrival times 
\begin{align*}
\widehat{t}_n :=
T_{Ln}-T_{L(n-1)}
\end{align*}
 with mean ${\mathbb E}\widehat{t}_n=La$, where the first customer arrives at time $T_L$ (see \cite{BaccFoss1994}, \cite{BaccFoss2004} or
\cite{FossKons2004} for further detail). Thus, $I_n$ is the time instant when the service of the $n$-th customer starts in UBQ. 

Due to the MP1, 
\begin{align}\label{eqn:xxt}
X_{1,L}^0 \le X_{1,L} \le T_L-T_1 + X_{1,L}^0 \quad \mbox{a.s.}
\end{align}

Further, let $S_{1,L}^{(k)}=\sum_{n=1}^{L}\sum_{i=1}^{N_{n,k}} \sigma_{n,k,i}$ 
be the total time server $k$ spends with service of first $L$ customers (again, thanks to the IP property, it is the same for all three auxiliary models introduced in this subsection). Then  $S_{1,L} = \sum_{k=1}^K S_{1,L}^{(k)}$.
Further, 
\begin{align}\label{2b}
\frac{S_{1,L}}{K}\leq  \max_k S_{1,L}^{(k)}\le X_{1,L}^0 \le S_{1,L} \quad \mbox{a.s.}
\end{align}
Here the second inequality follow since a server serves customers one by one, and the last inequality holds since at least one server is occupied at any time instant from time $0$ to time $X_{1,L}^0$  in the group-$L$ saturated model.

Clearly, $\dd{E} S_{1,L} = L\sum_{k=1}^K b_k \dd{E}N_k$. Further (see \cite{BaccFoss1994}), as $L\to\infty$,
\begin{align*}
\frac{X_{1,L}^0}{L} \to \max_{1\le k\le K} b_k \dd{E}N_k, \ \ \mbox{both a.s. and in the} \ {\cal L}_1 \ \mbox{norm}.
\end{align*}
Therefore, given \eq{stab}, one can choose $L$ so large that
\begin{align}\label{eq:defL1}
\dd{E} X_{1,L}^0 < La=\mathbb E T_{L}.
\end{align}
Baccelli and Foss (\cite{BaccFoss1994} and \cite{BaccFoss2004}) call $X_{1,L}^0$ and $X_{1,L}$ the ``maximal daters''.

One can see that, for any $L$ satisfying \eq{defL1}, the underlying Markov chain of the UBQ is Harris ergodic.
Namely, the sequence $\widehat{W}_n = I_n-T_{nL}$ (usually called the {\it workload process})  forms a Markov chain satisfying the so-called ``Lindley recursion'': 
\begin{align*}
W_0=W_1=0, \quad \widehat{W}_{n+1}=\max (0,\widehat{W}_n+
\widehat{\sigma}_{n}-\widehat{t}_{n+1}), \quad n\geq 1
\end{align*}
and is Harris ergodic\footnote{Recall that a Markov chain is Harris ergodic if it admits a unique stationary distribution, say $\pi$, and for any initial value its distribution converges to $\pi$ in the total variation norm.} if \eq{defL1} holds. Note that, similarly to the last inequality in  \eqref{2b},
we have $\widehat{\sigma}_n \leq S_{(n-1)L+1, nL}$ a.s., for any $n$.

Note also that $B$ and $\nu_B$ may be defined using the variables $X_{1,n}$, $n\ge 1$:
\begin{align}\label{eq:defnu}
\nu_B := \min \{n\ge 1 \ : \ X_{1,n}\le T_{n+1}-T_1\}
\end{align}
and then 
\begin{align}\label{eq:defB2}
B:=X_{1,\nu_B}.
\end{align}
In what follows, we use a simple observation (thanks to the second inequality in \eqref{eqn:xxt})
\begin{equation}\label{eqn:simpleb}
\text{if} \  T_L+X_{1,L}^0<T_{L+1}, \ \text{then} \
X_{1,L} < T_{L+1}-T_1, \ \ \text{thus} \  
\nu_B\leq L  \ \text{and} \  B\leq T_L-T_1+X_{1,L}^0.
\end{equation}
The above statement will help us in our construction of an upper bound for $B$.

\subsection{The scheme of the proof of Theorem \ref{th1}}\label{Sec3.4}

We start the proof with introducing in Section \ref{sec:ulbound} an
``extended'' busy period for the upper-bound single server queue (UBQ), whose length $U$ is a.s.\ bigger than $B$. Then in Section \ref{sec:PSBJ}, we justify the Principle of a Single Big Jump (PSBJ) for $U$ that says that $\{U>x\}$ occurs mainly due to the occurrence of a single unusual event $\Lambda_n(x)$ of a particular form.  

Define 
\begin{align}\label{eq:PSBJY2}
\begin{split}E_{n,k,i}(x, u_k):=
\left\{
\sum_{i=1}^{n-1}\widetilde{S}_i\leq h(x), 
 N_{n,k}\geq i, \widetilde{S}_n - \sigma_{n,k,i} \le h(x), 
 \sigma_{n,k,i} > u_kx\right\}\end{split}
\end{align}
 The $\sigma$'s and the $N$'s are introduced earlier for the customer-associated network infrastructure $\Sigma_C$ and the isolation model. Define also 
\begin{align}\label{eq:unionen}
\begin{split}&E_n(x,\{u_k\})\equiv E_n(x,\{u_k\}_{k \in [K]}):=\bigcup_{k=1}^K\bigcup_{i=1}^{\infty}E_{n,k,i}(x, u_k),\\
&E_n(x,c)\equiv E_n(x,\{c\}):=E_n(x,\{c\}_{k \in [K]}).
\end{split}
\end{align} 
Given that, we conclude in Section \ref{sec:PSBJ} 
that 
\begin{align}\label{eq:PSBJ5}
\{B>x\} = \{B>x, U>x\} \simeq \bigcup_{n\ge 1} 
\{B>x\}\cap \Lambda_n(x)
\end{align}
(see \eq{Ueq} for the expression of $\Lambda_n(x)$, and the above equivalence appears in \eq{btonub}), and the PSBJ for $B$ holds, too:  
\begin{align}\label{eq:PSBJ6}
\{B>x\} \simeq \cup_{n\geq 1} \{B>x, \nu_B\ge n\}\cap  \Lambda_n(x) \simeq \cup_{n\geq 1} \{B>x, \nu_B\ge n\}
\cap E_n(x, C_U),
\end{align}
(see \eq{utob} for the above equivalences) where $C_U$ is defined in Section \ref{sec:PSBJU}.

Then, in Section \ref{sec:finduk} we use an auxiliary deterministic fluid model with a single large random service time to  determine the ``correct'' values of $u_k$, such that  $u_k\ge C_U$ for any $k$ and 
\begin{align}\label{eq:PSBJX}
\begin{split}
\cup_{n\geq 1} \{B>x, \nu_B\ge n\}
\cap E_n(x, C_U)&\simeq \{B>x, \nu_B \ge n\}\cap E_n(x, \{u_k\}) \\
&\simeq \{\nu _B\ge n\}\cap E_n(x,\{u_k\}), \end{split}
\end{align}
and further
\begin{align}\label{eq:PSBJXX}
\{\nu _B\ge n\}\cap E_n(x,\{u_k\}) \simeq
\{\nu _B\ge n\}\cap \widehat{E}_n(x,\{u_k\})
\end{align}
where 
\begin{align*}
\widehat{E}_n(x,\{u_k\}) := 
\bigcup_{k=1}^K\bigcup_{i=1}^{N_{n,k}}
\{\sigma_{n,k,i}>u_kx\}.
\end{align*}
Here  
the (indicators of the) events 
$\widehat{E}_n(x,\{u_k\})$ form an $i.i.d.$ sequence and the event 
 $\{\nu _B\ge n\}$ does not depend on the events $\{\widehat{E}_n(x,\{u_k\}), \widehat{E}_{n+1}(x,\{u_k\}),\ldots.\}$, for any $n$. Therefore, we may continue the proof of Theorem \ref{th1} using first the Wald identity:
\begin{align*}
{\mathbb P} (B>x) &\sim
\sum_{n} {\mathbb P} (\widehat{E}_n(x,\{u_k\})) {\mathbb P} (\nu_B \ge n)\\
&=
{\mathbb P} (\widehat{E}_1(x,\{u_k\})) {\mathbb E} \nu_B
\end{align*}
and then by 
\begin{align*}
{\mathbb P} (\widehat{E}_1(x,\{u_k\}))\sim 
{\mathbb P} (E_1(x,\{u_k\})) &= 
\sum_{k=1}^K 
\sum_{i=1}^{\infty} {\mathbb P} 
(\sigma_{1,k,i}>u_kx, N_{1,k}\ge i, \widetilde{S}_1-\sigma_{1,k,i}\le h(x)) \\
&\sim
\sum_{k=1}^K 
\sum_{i=1}^{\infty} {\mathbb P} 
(\sigma_{1,k,i}>u_kx, N_{1,k}\ge i) \\
&=
\sum_{k=1}^K 
\sum_{i=1}^{\infty} {\mathbb P} 
(\sigma_{1,k,i}>u_kx) {\mathbb P} (N_{1,k}\ge i) \\
&=
\sum_{k=1}^K 
{\mathbb P} 
(\sigma_{k}>u_kx) {\mathbb E} N_{1,k}.
\end{align*}
The latter completes the proof of Theorem \ref{th1}, since $N_{1,k}$ and $N_k$ are identically distributed and \eq{ma} holds.

 \section{Upper and lower bounds for the first busy period length}
 \label{sec:ulbound}
 \setnewcounter

 \subsection{Upper bound}
 
 As we already noted, for a fixed network infrastructure, the length $B$ of the first busy period and the number $\nu_B$ of customers served within $B$ are non-monotone functions of arrival times. Therefore, the first busy period in the UBQ may be shorter than in the GJN itself (see again Example 2.1).
 However, we will show now that a certain random sum of busy periods in the UBQ dominates $B$.
 
 Recall that the workload process in the UBQ   
 is defined as  
\begin{align*}
\widehat{W}_0=\widehat{W}_1=0 \quad \mbox{and} \quad \widehat{W}_{n+1}=\max (0,\widehat{W}_n+
\widehat{\sigma}_{n}-\widehat{t}_{n+1}).
\end{align*} 
 Let $\widehat{\nu}_1$ be the number of new ``group customers'' served in the first busy period $\widehat{B}_1$ of this queue. We may introduce the events
 \[\widehat A_n:=\{ \widehat{W}_n+\widehat\sigma_n \le \widehat t_{n+1}=T_{L(n+1)}-T_{Ln} \},
 \] 
 to see that 
\begin{align*}
\widehat{\nu}_1 = \min \{ n\geq 1 \ : {\textbf 1}(\widehat A_n)=1\}  \quad
\mbox{and} 
\quad \widehat{B}_1 =\widehat \sigma_1+\ldots+\widehat\sigma_{\widehat\nu_1}. 
\end{align*}
Let $\widehat \nu_j$ be the number of group customers served in the $j$-th busy period and let $\widehat B_j$ be the length of the $j$-th busy period. Let $\widehat{\cal N}_0=0$ and $\widehat{\cal N}_j = \sum_{m=1}^j \widehat{\nu}_m$. Thus $\widehat{\cal N}_j$ is the total number of group customers in the first $j$ busy periods.  Then 
\[\widehat \nu_{j+1}=\min\{n\geq \widehat{\cal N}_{j}+1: \widehat W_n+\widehat\sigma_n\leq \widehat t_{n+1}\}-\widehat{\cal N}_j,\quad \widehat W_{\widehat{\cal N}_j+1}=0, \quad \widehat B_j=\widehat \sigma_{\widehat{\cal N}_{j-1}+1}+\ldots+\widehat \sigma_{\widehat{\cal N}_j}.\]
Note that $(\widehat \nu_j, \widehat B_j)$ are $i.i.d.$ Recall that $\widetilde{S}_k$ is the total service time of customer $k$ in the isolation model. One can see that, using \eqref{2b},   
\begin{align}\label{eqn:j=0}
B\leq S_{1,L},  
\end{align}
if there exists $j\in \{2,\ldots,L\}$ such that $S_{1,j-1}\le T_j-T_1$ 
(compare with \eqref{eqn:simpleb}). 

 Set $\widehat{\sigma}_0 =0$ by convention. 
Due to \eq{unbound}, for any $n\geq 0$,   
the probability $\dd{P}(A_n)$ of event
\begin{align*}
A_n :=  
 \bigcup_{i=1}^{L}\{\widehat{W}_n+\widehat\sigma_n+S_{Ln+1, Ln+i-1} \le 
T_{Ln+i}-T_{Ln} 
\}
\subseteq \widehat{A}_n
\end{align*}
is positive. Here $S_{Ln+1,Ln+i-1} $ is the total service time of customers $Ln+1$ to $Ln+i-1:$
\begin{align*}
S_{Ln+1,Ln+i-1} :=\sum_{i=Ln+1}^{Ln+i-1}\widetilde S_n= \sum_{k=1}^K \sum_{\ell=Ln+1}^{Ln+i} \sum_{j=1}^{N_{\ell,k}} \sigma _{\ell,k,j}, \quad \mbox{and}
\quad S_{Ln+1, Ln}=0, \quad \mbox{by convention}. 
\end{align*} 

One can see that, for any $m=1,2,\ldots$, group customer number $\widehat{\cal N}_{m-1}+1$ arrives in an empty system. Therefore, the event $A_{\widehat{\cal N}_m}$ is determined only by
\begin{align}\label{NN}
(\widehat{\sigma}_n, \widehat{t}_{n+1}), \ \ 
n=\widehat{\cal N}_{m-1}+1, \ldots, \widehat{\cal N}_m
\end{align}
and by 
\begin{align}\label{NNN}
(S_{L\widehat{\cal N}_m+1, L\widehat{\cal N}_m+i}, 
T_{L\widehat{\cal N}_m+i}-T_{L\widehat{\cal N}_m}), \ \ i=1,2,\ldots, L.
\end{align}
Note that the families of random variables in \eqref{NN} are i.i.d.\ for $m\geq 1$, and the following families of random variables  are $1$-dependent and identically distributed for $m\geq 0$: 
\begin{align*}
\{(\widehat{\sigma}_n, \widehat{t}_{n+1})\}_{ n=\widehat{\cal N}_{m}+1}^{\widehat{\cal N}_{m+1}}, \ \ 
\{(S_{L\widehat{\cal N}_m+1, L\widehat{\cal N}_m+i}, 
T_{L\widehat{\cal N}_m+i}-T_{L\widehat{\cal N}_m})\} _{i=1}^L. 
\end{align*}  
Then random variables 
$\{ {\bf 1}(A_{\widehat{\cal N}_m} ) \}_{m\geq 1}$ 
form a stationary sequence that is 1-dependent and, in particular, random variables
$\{ {\bf 1}(A_{\widehat{\cal N}_{2m}} ) \}_{m\geq 1}$ 
are $i.i.d.$
Let
\begin{align}\label{eqn:J}
J:=\min\{m\geq 0: {\bf 1}(A_{\widehat{\cal N}_m})=1\}+1  \ \text{and} \
J^*:=\min\{m\geq 1: {\bf 1}(A_{\widehat{\cal N}_{2m}})=1\}+1.
\end{align}
Since $p:=\mathbb P(A_{\widehat{\cal N}_m})>0$ due to  \eqref{eq:unbound}, random variable $J^*-1$ has a geometric distribution with parameter $p$, and since
$J\le 2J^*$ a.s., random variable $J$ has a light tail:
\begin{align}\label{JLT}
{\mathbb E} e^{cJ}<\infty, \ \ \text{for some}
\  c>0.
\end{align} 
Note that $J=1$ implies   \eqref{eqn:j=0}. 
We have now the following result. 

\begin{lemma}\label{upper_bound1}
We have 
\begin{align}\label{eq:upper1}
\begin{split}
B&\leq S_{1,L}+\sum_{j=1}^{J-1}(\widehat B_j+S_{
\widehat{ \cal N}_jL+1, (\widehat {\cal N}_j+1)L})\\
&=\sum_{j=0}^{J-1}(\widehat B_{j+1}+S_{
\widehat{ \cal N}_jL+1, (\widehat {\cal N}_j+1)L})\leq \sum_{j=0}^{J-1}(\widehat B_{j+1}+K\widehat B_{j+1})=\sum_{j=1}^{J}H_j=:U,\end{split}
\end{align}
where 
\begin{align}\label{eq:Hj}
\{H_j = (K+1)\widehat{B}_{j}\}_{j \geq 1} 
\ \ \text{is an $i.i.d.$ sequence.} 
\end{align}
We call $U$ the extended busy period. 
\end{lemma} 
\begin{proof}
The case for $J=1$ has been addressed by \eqref{eqn:j=0}. We will prove now the first inequality in \eqref{eq:upper1}  under the assumption $J\geq 2.$ Clearly, $J\ge 2$ implies that
\begin{align}\label{eqn:p1}
T_L-T_1\leq S_{1,L}.
\end{align}
The UBQ  
empties within time intervals $(T_{L(\widehat{\cal N}_i+1)}, T_{L(\widehat{\cal N}_{i+1}+1)})$ for any $i\geq 0$, and the original network
empties within such an interval if event $A_{\widehat{\cal N}_{i+1}}$ occurs. By the definition of $J$, the smallest $i$ for $A_{\widehat{\cal N}_{i+1}}$ to occur is $i=J-2$. Therefore, 
\begin{equation}\label{eq:b<} B\leq T_{L(\widehat {\cal N}_{J-1}+1)}-T_1. \end{equation}
Further,  for any $0\leq k<J-1,$ the event $A_{\widehat{\cal N}_{k}}$ does not occur, and  we have 
\begin{equation}\label{eqn:p2} T_{L(\widehat{\cal N}_{k+1}+1)}- T_{L(\widehat{\cal N}_k+1)}\leq \widehat B_{k+1}+S_{\widehat{\cal  N}_{k+1}L+1, (\widehat{\cal  N}_{k+1}+1)L}. \end{equation} 
The three above displays imply the first inequality in \eq{upper1}. 
For the second inequality in \eqref{eq:upper1}, we just need to 
note that
\begin{align*}
S_{\widehat{\cal N}_jL+1, (\widehat{\cal N}_{j}+1)L}\leq K\widehat B_{j+1}, \quad \forall j
\geq 0, \ \ \text{a.s.}
\end{align*}
Indeed, for any $k \in [K]$, the length $\widehat{B}$ of the busy period is bigger than the sum of all service times at station $k$ of all customers served within that period. Summing up these inequalities in all $k$ leads to the above inequality.
\end{proof}
 \begin{remark}\label{nub<}
From the definition of the extended busy period of UBQ, it is clear that $\nu_B\leq \widehat {\cal N}_JL$ a.s. 
 \end{remark}
 \subsection{Lower bound}

Now we present a lower bound for $B$. Since there is 
service of at least one customer during the busy period, the following simple lower bound clearly holds:
\begin{align}\label{eq:lb1}
B\ge \max_{1\le k \le K} \sum_{j=1}^{N_{1,k}}\sigma_{1,k,j} \ge \max_{1\le k \le K} {\bf 1}(N_{1,k}\ge 1)\sigma_{k,1},\quad \text{a.s}.
\end{align}
Therefore, there exists $c_0>0$ such that 
\begin{align}\label{eq:blower}
\dd{P} (B>x) \ge \dd{P}
(\max_{1\le k \le K} {\bf 1}(N_{1,k}\ge 1)\sigma_{k,1} >x) \ge 
 \max_{1\le k \le K} \dd{P}(N_{1,k}\ge 1) \dd{P} (\sigma_{k} >x)\geq  c_0\overline G(x).
\end{align}

\section{Principles of a Single Big Jump}
\label{sec:PSBJ}
\setnewcounter

\subsection{PSBJ for the extended busy period in the UBQ} 
\label{sec:PSBJU}

Recall notation $S_{1,n}^{(k)}$ from Section \ref{sec:auxiliary}, and the extended busy period length $U$ for the  single-server upper-bound queue (UBQ), where $U$ is $a.s.$ bigger than $B$. In order to formulate the PSBJ for the busy period $B$, we first present the PSBJ for $U$. We will need  the following lemma.

\begin{lemma}\label{BF2004}
 Assume that the conditions \eq{ma} are in force.  Take any increasing-to-infinity function $h(x)=o(x)$. 
Then,  for any $n=1,2,\ldots$,
\begin{align}
\label{updown}
\begin{split}
\{X_{1,n}^0>x\} & \simeq \bigcup_{j=1}^n\left\{\widetilde{S}_{j}>x\right\}\simeq \bigcup_{j=1}^n\left\{\widetilde{S}_{j}>x, \sum_{i=1}^{n}\widetilde{S}_i-\widetilde{S}_j\leq h(x)\right\}\\
&\simeq
\{ \max_{1\leq j \leq n} 
\max_{1\leq k \leq K}
\max_{1\leq i \leq N_{j,k}} \sigma_{j,k,i}>x\}.
\end{split}
\end{align}
Further, letting $\widehat G_n(x)=\mathbb P(X_{1,n}^0>x)$ and $\widetilde G(x)=\mathbb P(\widetilde{S}_1>x)$, we get 
\begin{align}
\label{eqn:l2-l1}
\widehat G_n(x)\sim n\widetilde G(x),\quad  \widetilde G(x)
\sim \ol{G}(x) \sum_1^K c_k \dd{E}N_k,  
\end{align}
In particular, the distribution of $X_{1,n}^0$ is ${{\sr{IRV}}}$, too. 
\end{lemma}

\begin{proof} The results follow from Lemma 8 in the Appendix of
\cite{BaccFoss2004}. For the sake of completeness, we recall here the reasoning. By \eq{exp}, $N_{n,k}$ has a light-tailed distribution. Moreover, it is independent of the $\{\sigma_{n,k,i}\}_{i\geq 1}$. 
By the sub-exponentiality of $\sigma_{n,k,i}$ 
and by the Kesten's bound (see, e.g., Theorem 3.34 in \cite{FKZ}), one can get  that, for any $k=1,\ldots,K$ with $c_k>0$, and for $n\ge 1$, 
\begin{align*}
\{ \sum_{i=1}^{N_{n,k}} \sigma_{n,k,i} >x\} \simeq
\{\max_{1\le i \le N_{n,k}} \sigma_{n,k,i}>x \}.
\end{align*}
The above display implies that, in our conditions, for any $k$ with $c_k>0$, the random variable $ \sum_{i=1}^{N_{n,k}} \sigma_{n,k,i}$ 
has an ${{\sr{IRV}}}$ distribution with the tail asymptotically equivalent to
$c_k\dd{E}N_k \ol{G}(x)$.
Then one uses the property of ${{\sr{IRV}}}$ distributions (that the sum of independent ${{\sr{IRV}}}$ random variables is again ${{\sr{IRV}}}$) to conclude that, for any $n$, 
\begin{align*}
\{\widetilde{S}_n>x\} \equiv
\left\{\sum_{k=1}^K \sum_{i=1}^{N_{n,k}} \sigma_{n,k,i} >x\right\} 
\simeq
\{\max_{1\le k \le K} \max_{1\le i\le N_{n,k}} \sigma_{n,k,i} >x\} 
\end{align*}
and, therefore, the second equivalence in \eqn{l2-l1} follows. Since $\max_{k} S^{(k)}_{1,n} \le X^{0}_{1,n} \le S_{1,n}$, see  (\ref{2b}), we get
\begin{align*}
  \{X_{1,n}^0>x\} & \simeq \{S_{1,n} > x\}  \simeq \left\{\max_{k} S^{(k)}_{1,n} > x \right\}.
\end{align*}
Hence, $S_{1,n}=\sum_{i=1}^n\widetilde{S}_i$ yields the first equivalence in \eqn{l2-l1} and the first ``$\simeq$'' in \eqref{updown}, whilst the second ``$\simeq$'' is obtained by \eq{newPSBJ2} of \lem{TARS} for $\gamma = n$.
\end{proof}

Now we are ready to formulate the PSBJ for the random variale $U$. Recall from \eq{Hj} the definition of random variables $H_j$. 
Based on 
Lemma \ref{BF2004} for $n=L$ and applying the result of
Proposition \ref{prop10} from Appendix \ref{sec:refined} for $\widehat{\sigma}$'s in place of 
$\sigma$'s,  $\widehat{t}$'s in place of $t$'s, $\widehat{\nu}$ in place of $\tau$, and $\widehat{\rho} := {\mathbb E} \widehat{t}_1/{\mathbb E} \widehat{\sigma}_1$, we may conclude that
\begin{align*}
\{H_1>x\}\equiv \{\widehat{B}_1>x/(K+1)\} 
\simeq 
\bigcup_{i\ge 1} \{\widehat{\nu}_1\ge i, \widehat{\sigma}_i > C_Ux \}
\end{align*}
and
\begin{align*}
{\mathbb P} (H_1>x) \sim L{\mathbb E} \widehat{\nu}_1 \ol{G}(C_Ux)\sum_{k=1}^{K} c_k {\mathbb E} N_k
\end{align*}
where \begin{align}\label{eq:c}C_U :=  (1-\widehat{\rho})/(K+1).\end{align}
 Recall that random variables $\{H_j\}_{j\geq 1}$ are $i.i.d.$ and that the number $J$ is a stopping time with respect to an extended filtration and has a light-tailed distribution, see \eqref{JLT}. Therefore, we can use 
 \lem{TARS} from Appendix \ref{auxib} (together with Remark \ref{rem:a1}), with 
$J$ in place of $\gamma$ and $H$
 in place of $\sigma$) ,
to conclude that
\begin{align*}
\{U>x\} & \simeq
\bigcup_{n=1}^J
\{H_n>x, \sum_{i=1}^{n-1} H_i \le h(x)\} \\
&\simeq
\bigcup_{i=1}^{\widehat{\cal N}_J}
\{\widehat{\sigma}_i> C_Ux, \sum_{\ell=1}^{i-1} 
\widehat{\sigma}_{\ell} \le h(x)\}.
\end{align*}
Here $h(x)=o(x)$ is such that $h(x)\to\infty$ as $x\to\infty$. 
Then, using Lemma \ref{BF2004} and an argument similar to that in \rem{a1}, we may continue, to obtain the following equivalence:
\begin{align}
\label{eq:Ueq}
\{U>x\} 
& \simeq \bigcup_{n=1}^\infty \Lambda_n(x):=
\bigcup_{n=1}^{\infty}
\{n\leq \widehat{\cal N}_JL, \widetilde{S}_n > C_Ux, \sum_{i=1}^{n-1} 
\widetilde{S}_i \le h(x) \}.
\end{align}
This is the PSBJ for $U$, and
the right-hand side of \eq{Ueq} is the union of events $Q_n$ that was mentioned earlier 
in Section \ref{Sec3.4}.

Further,  since the probability of the union of events on the right-hand side of \eq{Ueq} is not bigger than the sum of their probabilities and since
\begin{align*}
\{ \widetilde{S}_n > C_Ux, \sum_{i=1}^{n-1} 
\widetilde{S}_i \le h(x) \}  \subseteq 
\{\widetilde{S}_n > C_Ux \},
\end{align*}
we may conclude that by the Wald identity
\begin{align}\label{Ueq2}
{\mathbb P} (U>x) \leq (1+o(1))\mathbb E[\widehat{\cal N}_JL]\mathbb P(\widetilde S_1>C_Ux)=(1+o(1)) 
L{\mathbb E} J {\mathbb E} \widehat{\cal N}_1
{\mathbb P} (\widetilde{S}_1>C_Ux).
\end{align}

 \subsection{PSBJ for the busy period in GJN}
 \label{sec:PSBJB} 

 Since $B\leq U$ a.s., we get
 \begin{align*}
 \{B>x\} = \{B>x, U>x\}.
 \end{align*}
 It follows from   \eq{blower} and \eqref{Ueq2} and also from the properties of intermediately varying distributions that  
 the probabilities ${\mathbb P}(B>x)$ and ${\mathbb P}(U>x)$ are of the same order:  
 \begin{align*}
 \liminf_{x\to\infty}
 \frac{{\mathbb P} (B>x)}{{\mathbb P}(U>x)} >0.
 \end{align*}

Therefore (recall the notation $C_U=  (1-\widehat{\rho})/(K+1)$ and \eqref{eq:Ueq}),
\begin{align}\label{eq:btonub}
\begin{split}
\{B>x\} &\simeq \bigcup_{n=1}^\infty\{B>x\}\cap \Lambda_n(x)\\
&= \bigcup_{n=1}^\infty\{B>x, n\leq \nu_B\}\cap \Lambda_n(x)\\
&= 
\bigcup_{n=1}^{\nu_B}
\{ B>x, \widetilde{S}_n > C_Ux, \sum_{i=1}^{n-1} 
\widetilde{S}_i \le h(x) \}\\
&  \equiv 
\bigcup_{n=1}^{\infty} 
\{ B>x, \nu_B\ge n, \widetilde{S}_n > C_Ux, \sum_{i=1}^{n-1} 
\widetilde{S}_i \le h(x) \}.
\end{split}
\end{align}
Here the expressions in the second and the third lines are equal for all $x$ sufficiently
large since $B \leq \sum_{i=1}^{\nu_B} \widetilde{S}_i$, $\nu_B\leq \widehat {\cal N}_JL$, $h(x) < x$ for $x$ large, and, therefore, the events
\begin{align*}
 \{ B>x, \widetilde{S}_n > C_Ux, \sum_{i=1}^{n-1} 
\widetilde{S}_i \le h(x) \}
\end{align*}
are empty, for any $n > \nu_B$. 

Further, by Lemma \ref{BF2004} and its proof,  
\begin{align}\label{eq:utob}
\{B>x\} 
\simeq
\bigcup_{n\ge 1}
\{B>x, \nu_B\ge n\} \cap E_n(x, C_U),
\end{align}
see \eq{unionen} for $E_n(x, C_U)$. 
Thus, we have obtained 
\eq{PSBJ6}, that is the Principle of a Single Big Jump for the (first) busy period in the GJN.

We conclude this section with a number of comments that will be used in the next section.

Recall the notation $E_{n,k,i}(x,u_k)$ from \eq{PSBJY2} 
and
introduce the events 
\begin{align*}
\widehat{E}_{n,k,i}(x, C_U) \equiv \widehat{E}_{n,k,i}= 
\{N_{n,k}\ge i, \sigma_{n,k,i}>C_Ux\} \supseteq E_{n,k,i}(x,C_U).
\end{align*}
Then fix $n_0\ge 1$ and $i_0\ge 1$ and let
\begin{align*}
\Upsilon (n_0, i_0) = 
\bigcup_{n,k,i}^{(1)} E_{n,k,i}(x,C_U) \cap \{\nu_B\geq n\}
\end{align*}
where $\bigcup_{n,k,i}^{(1)}$ is the union over all triples $(n,k,i)$ such that $1\le k \le K, 1\le n \le n_0, 1\le i \le i_0$, 
and 
\begin{align*}
\Xi (n_0,i_0)  = 
\bigcup_{n,k,i}^{(2)}\widehat{E}_{n,k,i}(x,C_U) \cap \{\nu_B\geq n\}
\end{align*}
where $\bigcup_{n,k,i}^{(2)}$ is the union over all triples $(n,k,i)$ such that $1\le k  \le K$ and either $1\le i <\infty$ and $n>n_0$ or $1\le n < \infty$ and $i>i_0$. 
Then, using \eq{utob}, 
\begin{align}\label{eq:uplowB}
\Big(\{B>x\}\cap\Upsilon (n_0,i_0)\Big)
\stackrel{\sim}{\subset}
\{B>x\} 
\stackrel{\sim}{\subset}
\Big(\left(\{B>x\}\cap\Upsilon (n_0,i_0)\right) \cup \left(\{B>x\}\cap\Xi (n_0,i_0)\right)\Big).
\end{align}
Further,
\begin{align*}
{\mathbb P} (\Xi (n_0,i_0))
&\leq
\left(\sum_{n>n_0}\sum_{k=1}^K\sum_{i\ge 1} +
\sum_{n=1}^{n_0} \sum_{k=1}^K \sum_{i>i_0} \right)
{\mathbb P} (\nu_B\ge n) {\mathbb P} (N_{1,k}\ge i)
{\mathbb P}(\sigma_{1,k,1}>x)\\
&\leq (1+o(1))
\overline{G}(C_Ux)\left[{\mathbb E} (\nu_B, \nu_B >n)\sum_{k=1}^K c_k {\mathbb E} N_{1,k} + {\mathbb E} \nu_B \sum_{k=1}^K c_k {\mathbb E} (N_{1,k}, N_{1,k}>i_0)\right].
\end{align*}
Since ${\mathbb E} \nu_B$ and ${\mathbb E} N_k$, $k=1,\ldots,K$ are all finite, for any $\varepsilon >0$, one can choose $n_0$ and $i_0$ so large that the sum in the brackets becomes smaller than $\varepsilon$. Since any ${\sr{IRV}}$ distribution is dominated varying (see Appendix A), one may conclude that, for any $\varepsilon >0$, 
one can choose $n_0$ and $i_0$ so large that the probability ${\mathbb P} (\Xi (n_0,i_0))$ becomes asymptotically smaller than $\varepsilon \ol{G}(x)$. Taking into account \eq{blower} and \eq{uplowB}, we conclude that 
for any $\varepsilon>0,$ if we choose $n_0,i_0$ large enough, then for large $x$,  we have 
\begin{align}\label{eq:cutshort}
\begin{split}
\mathbb P(B>x)\leq (1+\varepsilon)\mathbb P(\{B>x\}\cap\Upsilon (n_0,i_0)),
\end{split}
\end{align}
where we recall 
\begin{align}\label{eq:shortpart}
&\{B>x\}\cap\Upsilon (n_0,i_0)=
\bigcup_{n,k,i}^{(1)} \{B>x, \nu_B\geq n\}\cap E_{n,k,i}(x,C_U).
\end{align}
In Section \ref{sec:finduk}, we will use the above two displays to prove \eqref{eq:PSBJX}. This will lead to the completion of the proof of Theorem \ref{th1}.

\section{Completion of the proof of Theorem \ref{th1} based on fluid limits}\label{sec:finduk}

\subsection{Verification of \eq{PSBJX}}
\label{sec:fluid}
\setnewcounter

For any $c>0,$ define 
\[E_{n,k,i}(x,=c):=\left\{N_{n,k}\geq i,
\sum_{j=1}^{n-1}\widetilde{S}_j\leq h(x), 
 \widetilde{S}_n - \sigma_{n,k,i} \le h(x), 
 \sigma_{n,k,i} =cx \right\}.\]
 We will show that there exists $x_{k,c}\xrightarrow[]{x\to\infty}\infty$ such that $0<\overline x_{k,c}:=\lim_{x\to\infty}\frac{x_{k,c}}{x}$ and 
\begin{align}
\label{eq:toproveinSec6.1}
&\mathbb P\left( \left. \left|\frac{B}{x_{k,c}}-1\right|\geq \varepsilon\, \right|\, \nu_B\geq n, E_{n,k,i}(x,=c) \right)\xrightarrow{x\to\infty}0,\quad \forall \varepsilon>0.
\end{align}
Moreover, we will show that $\overline x_{k,c}$ is a strictly increasing and linear function of $c$ (see \eqref{eq:barxkclinear}).  Then we obtain that 
\begin{align}\label{eq:ctoxkc}\begin{split}&\{B>x_{k,c}, \nu_B\geq n\}\cap E_{n,k,i}(x,c)\simeq \{\nu_B\geq n\}\cap E_{n,k,i}(x,c).
\end{split}\end{align}
Indeed, the above is due to {\bf MP2} and the definition of $E_{n,k,i}(x,c)$, as they, together with \eqref{eq:toproveinSec6.1}, entail that 
\[\mathbb P(B>x_{k,c}\,|\, \{\nu_B\geq n\}\cap E_{n,k,i}(x,c))\xrightarrow[]{x\to\infty}1.\]

Since $\overline x_{k,c}$ is a strictly increasing and linear function of $c$ and given \eq{utob}, we can conclude that 
\[1\geq \overline x_{k,C_U}, \quad \forall k \in [K]. \]
The above display allows us to find $u_k\geq C_U$
such that 
\[
1=\overline x_{k,u_k}, \quad \forall k.
\]
Then, using \eq{toproveinSec6.1} and \eq{ctoxkc},  we obtain that (in the same spirit of obtaining \eqref{eq:ctoxkc})
\begin{align}\label{eq:removeB>x}\begin{split}
&\{B>x, \nu_B\geq n\}\cap E_{n,k,i}(x,C_U)\simeq\{B>x, \nu_B\geq n\}\cap E_{n,k,i}(x,u_k)\simeq \{\nu_B\geq n\}\cap E_{n,k,i}(x,u_k).
\end{split}\end{align}
Taking into account \eq{cutshort} and \eq{shortpart} and  the fact
\begin{align}\label{eq:disjoint}
\text{events} \  E_{n,k,i}(x,c)  \text{ are disjoint for different collections of triples} \ n,k,i,
\end{align} 
we arrive at \eq{PSBJX}.  To complete the proof of \eq{PSBJX},  we just need to prove \eq{toproveinSec6.1} and the existence of $x_{k,c}$ and  $\overline x_{k,c}$ being a strictly increasing and linear function of $c$. 
This will be achieved by the fluid approximation, as shown below. 

Take any fixed $n, i$ and $k$ with $c_k>0$.
Recall random variable $\sigma_k$ has an {$\sr{IRV}$} distribution. 
We rely on the overview on fluid limits presented in Appendix E.

We call the service with service time $\sigma_{n,k,i}$ a {\it special service}.  To simplify the notation, assume that the first exogeneous customer arrives at time $T_1=0$. 
We are conditioning on values of $\widetilde{S}_j, j<n$, 
$\widetilde{S}_n-\sigma_{n,k,i}$ and $T_j, j\le n$ such that $\nu_B\ge n$ and the event $E_{n,k,i}(x,=c)$ takes place, for some $c>0$. Given that, the special service starts at some time $T^{(S)}$ (here $S$ for special) that occurs after time $0$ and before time $2h(x)$. During $2h(x)$ units of time, a random number of customers, say $L$, do arrive where
$L$ is of order at most $h(x)$ with high probability, i.e. 
\begin{align*}
{\mathbb P} (L \le L_0h(x))\to 1 \ \ \text{as} \ x\to\infty,
\end{align*}
with $L_0=2(1+\varepsilon)/a$, and the residual inter-arrival and residual service times at all stations but $k$ are of order $h(x)$ with high probability, too -- thanks to the basic renewal theory.

For simplicity, we condition on the value, say $x$, of the special service time (i.e.\ $c=1$ in $E_{n,k,i}(x,=c)$). We may view this full service time as a ``residual service time''. Then we are exactly in the setting of Subsection E.2 and may conclude that
\begin{align*}
|Z(\widetilde{T})|/x \to 0 \ \text{in probability},
\end{align*}
where $Z(\cdot)$ is defined in \eqref{eq:zt}, $\widetilde{T} = T^{(S)}+ \tau^{(k)}x$ and $\tau^{(k)}$ is given in \eq{tauk}, 
and this implies that, for some function $\widetilde{h}(x)$ that is an $o(x)$ function,
\begin{align*}
{\mathbb P}(|Z(\widetilde{T})|\le \widetilde{h}(x))\to 1, \ \ \text{as} \ x\to\infty.
\end{align*}

Now we are in the setting of \lem{laststep} with $y=\widetilde{h}(x)$ and may conclude that,  
\begin{align*}
{\mathbb P}(B\in (\tau^{(k)}x- \widehat{h}(x), \tau^{(k)}x+\widehat{h}(x))) \to 1, \ \text{as} \ x\to\infty,
\end{align*}
for some function $\widehat{h}(x)=o(x)$. Then  recalling $x_{k,1}$ in \eq{toproveinSec6.1}, we can set
\[x_{k,1}=\tau^{(k)}x.\]
This entails $\overline x_{k,1}=\tau^{(k)}$. Recall that we have set $c$=1. For general $c$,  
the corresponding term for $\tau^{(k)}$ will be $\overline x_{k,c}=c\tau^{(k)}$, since the fluid model is linear. 
Therefore, we have proved the existence of $x_{k,c}$ and 
\begin{equation}\label{eq:barxkclinear}\overline x_{k,c}=c\tau^{(k)},\end{equation}
which is strictly increasing and linear on $c$, and \eq{toproveinSec6.1} is proved. 
This leads to $u_k$ which solves the equation: 
\begin{align}\label{eq:def-uk}
\overline x_{k,u_k}=1, \ \ \text{therefore} \ \  u_k =1/\tau^{(k)}.
\end{align}
We can conclude now that the proof for Theorem \ref{th1} is completed with $u_k$ given above.  However, equality \eq{def-uk} determines $u_k$ implicitly, and in the next subsection we provide an algorithm for finding its value.

\subsection{The algorithm for determining $u_k$}\label{sec:algo}

Consider the fluid model described in Appendix \ref{app:pc},. 
The dynamics therein starts (at time $0$) with all servers empty, and server $k$ ``frozen'' for the time duration 1 (in the sense that server $k$ has input but no output of fluid), which represents an unusually long service time  
in the original model, and $\tau^{(k)}$ is the total time from time $0$  until the fluid model becomes empty. 
The time period $[0,\tau^{(k)})$ is naturally split into two time periods: the first period is $[0,1)$ when the server $k$ is frozen, and the second period, between time instants $1$ and $\tau^{(k)}$, when all stations are operating.

We will use arrival and departure rates to/from all servers to describe the dynamics in both periods. Recall $[K] = \{1,2,\ldots,K\}$. 

Let us start with the first period. Recall the notation:
$\overline a_j^{(k)}(t)$ (resp.\ $\overline d_j^{(k)}(t)$) is  the arrival (resp.\ departure) rate to/from station $j$ for any $t\geq 0$. Note that $\overline a_j^{(k)}(t), \overline d_j^{(k)}(t)$ for $j=1,2,\ldots, K$ are constants within the period and satisfy the equations \eq{a=d} in \app{fluid-limit}, which has a unique solution.  Using \eq{ratek}, we know that at time instant $1$ all servers are empty except server $k$ that  has fluid volume of size 
$\frac{\beta_k}{a}$ where $\beta_k={\mathbb P} (N_k>0)>0$.

From time $1$, the model enters the second period. For its analysis, we split the period into two consecutive time intervals.

\noindent{Interval 1:} This time interval starts at time $1$ and ends when server $k$ becomes empty from fluid. In this time interval, there are some servers that increase linearly their fluid volumes at constant rates, and others remain empty (except server $k$ which decreases its volume linearly).  

\noindent{Interval 2:} It starts when server $k$ becomes empty and contains at most $K-1$ subintervals, and each of them starts  when one or more non-empty servers become(s) empty simultaneously (and will remain empty after that). We determine the increasing or decreasing rates for all the servers with positive fluid levels during each such time subinterval. 

We analyse now  

\noindent{\bf Interval 1.}  
We provide an algorithm for finding the set of servers that stay free of fluid during this time interval. 

{\bf Description of the algorithm.} We start with the guideline as follows, before proceeding to more details. 
 We use $\mathcal A$ to denote the set of servers that stay free of fluid. 
The algorithm finds $\mathcal A$ using induction arguments. Here is the scheme of induction. We start with $\mathcal A_0 =\emptyset$ at the initial step $0$ and, for any induction step $r$, given $\mathcal A_r$, we determine a set $\widehat{\cal A}_r \subset [K]\setminus {\mathcal A}_r$, where recall $[K] = \{1,2,\ldots,K\}$.  If
$ \widehat{\cal A}_r  \neq \emptyset$, then we let
$\mathcal A_{r+1} = \mathcal{A}_r \cup \widehat{\cal A}_r $ and go to step $r+1$. Otherwise, if $ \widehat{\cal A}_r$ is empty, we let $\mathcal A = \mathcal{A}_r$ and stop the algorithm.

Now we describe in detail a general step $r$ of the algorithm. We know that during this time interval the server $k$ decreases its fluid level, and the departure rate is at the maximal level $1/b_k$.   Assume that
the set of  servers that stay free of fluid is $\mathcal{A}_r$.  
Under this assumption, 
for any server $j\in{\mathcal A}_r$, the arrival $\overline{a}_j^{(k)}(1)$ and the departure $\overline{d}_j^{(k)}(1)$ rates must be equal,  
 whereas 
for any server $\ell\in [K]\setminus ( {\mathcal A}_r\cup \{k\})$, the arrival rate must be strictly larger than the departure rate which takes the maximal value $1/b_{\ell}$. For simplicity,  we use $X_j$ to denote the arrival rate of server $j$. Then the following inequalities should hold: 
\begin{equation}\label{eqn:part1}\forall j\in \mathcal A_r,\quad p_{0,j}/a+\sum_{\ell \in [K]\backslash \mathcal{A}_r}p_{\ell,j}/b_{\ell}+\sum_{s\in 
\mathcal{A}_r}p_{s,j}X_s=X_j;\end{equation}
\begin{equation}
\label{eqn:part2}
 \forall \ell  \in[K]\setminus ( {\mathcal A}_r\cup \{k\}),\quad p_{0,\ell}/a+\sum_{s\in [K]\backslash \mathcal A_r } p_{s,\ell}/b_{s}+
 \sum_{j\in \mathcal{A}_r}p_{j,\ell}X_j>1/b_{\ell}.
 \end{equation}
Note that \eqref{eqn:part1} uniquely determines $\{X_i\}_{i\in \mathcal A_r}$. Then we define $\widehat {\mathcal A}_r$ as 
\[\widehat {\mathcal A}_r:=\Big\{\ell\in 
[K]\backslash(\mathcal{A}_r\cup\{k\}): \text{the inequality in \eqref{eqn:part2} does not hold}\Big\}.\]
One can see that
\begin{itemize}
\item If $\widehat{\mathcal{A}}_r=\emptyset$, then there is no contradiction, which means that all servers in ${\mathcal A}_r$ remain free of fluid, and all servers in
$[K]\backslash\{{\mathcal A}_r\cup \{k\}\}$ increase their fluid volumes.  
\item If $\widehat{\mathcal{A}}_r\neq \emptyset$, then we get a contradiction, which means that not all servers in $[K]\backslash\{{\mathcal A}_r\cup\{k\}\}$ increase their fluid volumes. Therefore, using the monotonicity argument, we may conclude that all servers  in ${ {\mathcal A}_r}\cup \widehat {\mathcal A}_r$ must stay free of fluid. Therefore,  we set 
${\mathcal A}_{r+1}={\mathcal A}_r\cup \widehat {\mathcal A}_r$ and turn to step $r+1$.
\end{itemize}

Denote the duration of the time interval 1 by $\mathcal T_1$.
Since the initial volume of fluid at server $k$ is $\beta_k/a$, the arrival rate is $\overline{a}_k^{(k)}(1)$ 
and the departure rate is $1/b_k$, we get 
\begin{align}\label{eq:t1}\mathcal T_1 =\frac{\beta_k/a}{1/b_k-\overline{a}_k^{(k)}(1)}.\end{align}
Note that $\overline{a}_k^{(k)}(1)$ can be obtained by solving \eqn{part1} and \eqn{part2} with $\mathcal A_r=\mathcal A$, and then $\overline{a}_k^{(k)}(1)=X_k.$

Now we consider {\bf Interval 2} of the second period, that starts at time 
$1+\mathcal T_1$. This interval can be decomposed into several subintervals, see (p10) 
 in Proposition \ref{propE3}, each of which is an interval between two consecutive events where one or more non-empty servers become empty.  To find the increase/decrease rates of fluid at any station in any of the subintervals and their lengths, we follow the property (p13) from Proposition \ref{propE3}. Denoting the length of time interval 2, which is the sum of the lengths of all subintervals, by $\mathcal T_2$, we may conclude that 
the total time to empty the fluid model is 
\begin{align}\label{eq:tauk}\begin{split}\tau^{(k)}&=1+\text{time interval 1 length}+\text{time interval 2 length}\\
&=1+\mathcal T_1+\mathcal T_2=1+\frac{\beta_k/a}{1/b_k-\overline a_k^{(k)}(1)}+\mathcal T_2.\end{split}\end{align}

In the next subsection, we consider an example with two servers where the values of $\tau^{(k)}$, $k=1,2$ may be found explicitly.

\vspace{0.3cm}
\subsection{Example of 2-server network}
\label{sec:example}

We consider a particular case of two stations.
To simplify the formulae, we may assume that 
$p_{k,k}=0$, for both
$k=1,2$ (if this is not the case, we may introduce new service times that
are geometric sums of original service times in the original model).

Here the stability condition \eq{stab} consists of
\begin{align}\label{eq:server1}
(p_{0,1}+p_{0,2}p_{2,1})b_1 <a (1 - p_{1,2} p_{2,1})
\end{align}
and
\begin{align}\label{eq:server2}
(p_{0,2}+p_{01}p_{1,2})b_2 <a (1 - p_{1,2} p_{2,1}).
\end{align}
As before, we proceed with the analysis at the fluid level, but more informally.

By the symmetry, it is enough to consider only   the case where the first server is frozen for the unit time. 
There are two time periods. The first period is $[0,1)$ when server 1 has no output. At time $1$, server 2 is empty and server 1 has fluid volume 
\[ 
\frac{\beta_1}{a}=\frac{p_{0,1}+p_{0,2}p_{2,1}}{a},
\] 
where 
$\beta_1=\mathbb P(N_1>0)=p_{0,1}+p_{0,2}p_{2,1}.$

The second period is from time $1$ until the moment when the fluid model becomes empty. It includes two time intervals, with respective time lengths $\mathcal T_1$ and $\mathcal T_2.$ Note that the arrival rate to server 2 within interval 1 is  
\[\overline a_2^{(1)}(1)=\frac{p_{0,2}}{a}+\frac{p_{1,2}}{b_1}.\] 
Since we have only two servers, we can analyse directly two simple cases (without referring to the general algorithm from the previous subsection)

\begin{enumerate}
\item  (Server 2 remains free of fluid during time interval 1). If $1/b_2$, the departure rate from server 2, is
at least as big as the arrival rate, 
\begin{align}\label{eq:v}
1/b_2 \ge \overline a_2^{(1)}(1),
\end{align}
then
server 2 remains empty and server 1 receives fluid 
at rate
\begin{align*}
\overline a_1^{(1)}(1)=\frac{p_{0,1}}{a} + \overline a_2^{(1)}(1)p_{2,1},
\end{align*}
which is smaller than the departure rate $1/b_1$ from server  1, by the stability condition \eq{server1}. 
Therefore the fluid level at server 1 decreases at rate
\begin{align*}
\beta_{1,1} := \overline d_1^{(1)}(1)-\overline a_1^{(1)}(1)=  \frac{1}{b_1}-\overline a_1^{(1)}(1)= \frac{1}{b_1} - \frac{p_{0,1}}{a} - \left(\frac{p_{0,2}}{a}+\frac{p_{1,2}}{b_1}\right)p_{2,1}
 > 0, 
\end{align*}
and then 
\[
\mathcal T_1=\frac{\beta_1/a}{1/b_1-\overline a_1^{(1)}(1)}=\frac{\frac{p_{0,1}+p_{0,2}p_{2,1}}{a}}{\frac{1}{b_1}-\frac{p_{0,1}}{a}-(\frac{p_{0,1}}{a}+\frac{p_{1,2}}{b_1})p_{2,1}}.
\]
In this case, at time $1+\mathcal T_1$, both servers are empty, thus $\mathcal T_2=0$. Therefore, 
\[
\tau^{(1)}=1+\mathcal T_1=1+\frac{\frac{p_{0,1}+p_{0,2}p_{2,1}}{a}}{\frac{1}{b_1}-\frac{p_{0,1}}{a}-(\frac{p_{0,1}}{a}+\frac{p_{1,2}}{b_1})p_{2,1}}.
\]

\item (Server 2 increases its fluid level during time interval 1). Assume now that 
\begin{align}\label{eq:vv}
1/b_2  <  p_{0,2}/a + p_{1,2}/b_1.
\end{align}
Then, starting from time $1$, the fluid level at server 2 grows with rate
\begin{align*}
\gamma_{2,1} :=\overline  a_{2}^{(1)}(1)-\overline d_{2}^{(1)}(1)=p_{0,2}/a + p_{1,2}/b_1 - 1/b_2,
\end{align*}
and server 1 receives fluid at rate
\begin{align*}
\overline a_1^{(1)}(1)=p_{0,1}/a  + p_{2,1}/b_2,
\end{align*}
so 
\[\mathcal T_1=\frac{\beta_1/a}{1/b_1-\overline a_1^{(1)}(1)(1)}=\frac{\frac{p_{0,1}+p_{0,2}p_{2,1}}{a}}{\frac{1}{b_1}-\frac{p_{0,1}}{a}-\frac{p_{2,1}}{b_1}}.\]
Then the fluid level at server 2 at time $1+\mathcal T_1$ is 
\[\gamma_{2,1}\mathcal T_1.\]
Next we consider time interval 2 which starts at $1+\mathcal T_1$ and finishes when the fluid model becomes empty. 

Starting from time $1+\mathcal T_1$, server 1 stays empty, so the arrival
rate to server 2 is 
\begin{align*}
\overline a_2^{(1)}(1+\mathcal T_1)=p_{0,2}/a  +p_{1,2}\overline a_1^{(1)}(1+\mathcal T_1)=p_{0,2}/a  + p_{1,2}(p_{0,1}/a +p_{2,1}/b_2)
\end{align*}
(here we have used that $\overline a_1^{(1)}(1+\mathcal T_1)=\overline a_1^{(1)}(1)$), and the fluid level at server 2 decays at rate
\begin{align}\label{eq:gamma22}
\beta_{2,2}:= \frac{1}{b_2} -\overline a_2^{(1)}(1+\mathcal T_1) = \frac{1}{b_2} - \frac{p_{0,2}}{a} - 
\left(\frac{p_{0,1}}{a}+\frac{p_{2,1}}{b_2}\right)p_{1,2} >0.
\end{align}
Thus, starting from time $1+\mathcal T_1$, it takes time 
\begin{align}\label{eq:y3}
\frac{\gamma_{2,1}\mathcal T_1}{\beta_{2,2}}
\end{align}
to empty server 2. When server 2 is empty, the fluid model is empty.  Thus, 
\[\mathcal T_2=\frac{\gamma_{2,1}\mathcal T_1}{\beta_{2,2}}.\]
In conclusion,
\[\tau^{(1)}=1+\mathcal T_1+\mathcal T_2=1+(1+\frac{\gamma_{2,1}}{\beta_{2,2}})\mathcal T_1.\]
\end{enumerate}

\section{Prospective generalisations}
\label{sec:concluding}

In this paper, we consider two ways of presentation for the GJN's, via the server-based infrastructure $\Sigma_S$ or the  customer-based infrastructure $\Sigma_{C}$. In this section, we make several comments on possible generalisations of the stochastic assumptions on the model using the $\Sigma_C$ framework. 

For example, 
for each fixed $n$,  one may assume that the service times $\{\sigma_{n,k,i}; k \in [K], i \ge 1\}$ are dependent, but
the families are $i.i.d.$ in $n$. More generally, one may assume that the pairs of families
$\{(\sigma_{n,k,i}, \alpha_{n,k,i}); k \in [K], i \ge 1\}$ are $i.i.d.$ in $n=1,2,\ldots$, but there may be some dependence within each family. Some dependence between the duration of a service time and a further routing of the customer looks realistic in certain applications. 

It was shown in Baccelli and Foss \cite{BaccFoss1994} that the stability conditions stay the same in a much more general setting, where the pairs of families $\{(\sigma_{n,k,i}, \alpha_{n,k,i}); k \in [K], i \ge 1\}$ form a stationary and ergodic sequence in $n=1,2,\ldots$, and any dependence within each family is allowed. In particular, we expect that if one assumes the families to be $i.i.d.$ in $n$, then, under some  constraints on the dependence structure within a family, one can use the same approach for establishing a Big-Jump Principle (that the number of jumps may be bigger than one, in general) at the level of $\widetilde{S}_i$, and then proceed to the analysis of fluid models, with obtaining various tail asymptotics. It was provied in \cite{BaccFoss1994} that the monotonicity and the invariance properties continue to hold in the more general setting, and then one can construct the single-server upper-bound queue and, further, our construction in \sectn{ulbound} continues to work. 
Thus, one can establish similar results, under some relaxed assumptions on the dependence structure. 

Let us provide two particular examples.

\begin{example}
\label{exa:}
For any fixed $n \ge 1$ and $k \in [K]$,  $\sigma_{n,k,i} = \sigma_{n,k,1}$ a.s. for all $i \ge 1$. This means that, say, in the isolation model all service times of each customer at any fixed station are identical. Further, we may assume that $\{\sigma_{n,k,1}; n \ge 1, k \in [K]\}$ are $i.i.d.$ and independent of all $\alpha$'s.
\end{example}

\begin{example}
\label{exa:}
We may assume a pairwise dependence between a service time and the following outgoing link. 
Say, let $C>0$ be a number.
If $\sigma_{n,k,i} \le C$, then $\alpha_{n,k,i}$ may take any value from 1 to $K+1$, and if
$\sigma_{n,k,i} > C$, then $\alpha_{n,k,i} = K+1$.
(If a customer takes a very long service time, it has to leave the network).
\end{example}

The other direction of generalisation is to consider the heavy-tail asymptotics for the length of the first busy period when the heaviest service time distributions belong to the more general class of square-root insensitive distributions. 
Under certain further technical assumptions, the same scheme of the proof works for this class of distributions. 
\appendix

\section*{Appendix}
\label{sec:Appendix}
\setnewcounter

\section{Basic properties of ${\sr{IRV}}$ and related classes of distributions}\label{auxia}
\setcounter{section}{1}

\setnewcounter

As it follows from property \eqref{IRV-1}, any ${\sr{IRV}}$ distribution $G$
is
{\it long-tailed}, i.e.
 $\ol{G}(x+y) \sim \ol{G}(x)$, for any $y>0$. Further, any ${\sr{IRV}}$ distribution is
{\it subexponential}, i.e. 
 $\ol{G*G}(x) \sim 2\ol{G}(x)$. 

Also, any ${\sr{IRV}}$ distribution is {\it dominated-varying}, i.e. there exists a positive constant $c=c(G)$ such that $\ol{G}(2x) \geq c \ol{G}(x)$, for all $x>0$.
Since $\ol{G}(x)$ is a non-increasing function, this implies that the functions $\ol{G}(x)$ and $\ol{G}(dx)$ are of the same order, for any $d>0$.

{\it Regularly varying distributions} ($\sr{RV}$) form a subclass of $\sr{IRV}$ distributions. Recall that  $G\in \sr{RV}$ if, for some 
$\alpha >0$, 
\begin{align}
\label{eq:RV}
  \ol{G}(x) = x^{-\alpha} L(x),
\end{align}
where $L(x)$ is a {\it slowly varying} (at infinity) function, i.e. $L(cx)\sim L(x)$ as
$x\to\infty$, for any $c>0$.

If a random variable $X$ has an ${\sr{IRV}}$ distribution, then for any constant $C>0$, a random variable $CX$ has an ${\sr{IRV}}$ distribution, too. 

If a random variable $X$ has an ${\sr{IRV}}$ distribution $G$ and the distribution of a random variable $Y$ is such that ${\mathbb P} (|Y|>x) = o(\overline{G}(x))$, as $x\to\infty$, then
\begin{align}\label{eq:oG}
{\mathbb P} (X+Y >x) \sim \overline{G}(x), \ \ x\to\infty,
\end{align}
where any dependence between $X$ and $Y$ is allowed.
If $X$ and $Y$ are mutually independent, then  \eq{oG} holds in a more general setting where $G$ is subexponential and ${\mathbb P} (Y>x) = o(\overline{G}(x))$. 

 More properties and details may be found, e.g., in the book \cite{FKZ}. 

\section{Random sums}\label{auxib}
\setcounter{section}{2}
Let $\sigma_1,\sigma_2,\ldots$ be an $i.i.d.$\ sequence of non-negative random
variables having a common distribution function $G$ with finite mean $b=\dd{E}\sigma_1$. Let $S_n=\sum_{i=1}^n\sigma_i.$ Let $\mathcal R=(\mathcal R_n)_{n\geq 1}$ be an arbitrary filtration such that $(S_n)_{n\geq 1}$ is adapted to $\mathcal R$, i.e.  $\{\sigma_i\}_{1\le i \le n}$ are measurable with respect to ${\mathcal R}_n$ and 
$\{\sigma_i\}_{i>n}$ do not depend on ${\mathcal R}_n$, for any $n$. Let $\gamma$ be any stopping time w.r. to $\mathcal R$. 
The lemma below is an application of \cite[Theorem 10]{denisov2010asymptotics}. 

\begin{lemma}\label{lem:TARS}
Let $h(x)$ be any increasing-to-infinity function and $h(x)=o(x)$. Let $\gamma$ be any stopping time w.r. to $\mathcal R$.  Assume that $G\in {{\sr{IRV}}}$ and that  $\gamma$ has a light-tailed distribution. Then, 
as $x\to\infty$,
\begin{align}
\label{eq:newPSBJ2}
\{S_{\gamma}>x\} \simeq \bigcup_{n=1}^{\gamma} 
\{\sum_{1\leq i\leq n-1}\sigma_i\leq h(x),  \sigma_n >x\} 
,
\end{align}
and
\begin{align}\label{eq:TARS}
\dd{P}(S_{\gamma}>x) \sim \dd{E}\gamma \ol{G}(x).
\end{align}
\end{lemma}
\begin{remark}\label{rem:a1}
Formally, in \cite[Theorem 10]{denisov2010asymptotics}, only \eq{TARS} was justified. However,   \eq{newPSBJ2} follows directly from the inclusion below: 
\begin{align*} 
\{S_{\gamma}>x\} \supseteq 
\bigcup_{n=1}^{\gamma} \{\sum_{1\leq i\leq n-1}\sigma_i\leq h(x),  \sigma_n >x\} \equiv 
\bigcup_{n=1}^{\infty} \{\gamma\ge n, \sum_{1\leq i\leq n-1}\sigma_i\leq h(x),  \sigma_n >x\}.
\end{align*}
Since the events in the second union  are disjoint, the probabilities of the above three events have the same asymptotics as the right hand side in \eq{TARS}. 

\end{remark}

\section{Tail asymptotics for the busy period in a stable single-server queue}
\label{sec:refined}
\setnewcounter

The following result may be found in \cite{Zwar2001}   and in \cite{FossMiya2018} (see Theorem 3.1 and Corollary 3.1 therein). see also a multivariate version in \cite{AsmuFoss2018}. 

Let $\{\sigma_n\}$ and $\{t_n\}$ be two mutually
independent sequences of non-negative random variables and let each of them be $i.i.d.$, with finite means $\dd{E}\sigma_n =b$
and $\dd{E} t_n =a$. Assume $\rho :=b/a <1$. 
Let $\xi_n=\sigma_n-t_n.$ Let $V_0=0$ and
$V_n = \sum_{k=1}^n \xi_k$ and $S_n = \sum_{k=1}^n \sigma_k$. Let $\tau = \min \{n\ge 1 \ : \ V_n \le 0\}$ 
and $B=S_{\tau}$. 

Let $\sigma_n$ have distribution $G$ from the class ${\sr{IRV}}$. An inspection of the proof of Theorem 3.1 in \cite{FossMiya2018} shows than, in fact, the following 
(formally, more general) result holds:

\begin{proposition}\label{prop10}
As $x\to\infty$,
\begin{align}\label{eq:joint}
\{B>x\} \simeq \{B>x, \tau > x/a\} \simeq \{\tau > x/a\} \simeq \cup_{n\ge 1} \{\tau \ge n, \sigma_n >x(1-\rho)\}.
\end{align}
\begin{align*}
\sum_{m<n} \dd{P}(\tau\ge n, \sigma_m>x(1-\rho), \sigma_n>x(1-\rho)) = o(\dd{P}(B>x))
\end{align*}
and, therefore, 
\begin{align} \label{eq:B2}
\dd{P} (B>x) \sim \dd{E} \tau \overline{G}(x(1-\rho ))
\quad \mbox{and}
\quad
\dd{P} (\tau >n) \sim \dd{E} \tau \overline{G}(n(a-b)),
\end{align}
as $x\to\infty$ and $n\to\infty$.

\end{proposition}

\section{Upper bounds on the time to empty the network}
\label{Upper-empty}
\setnewcounter

This section contains two simple lemmas, and their proofs are based on the first principles.

\begin{lemma}\label{lem:maximaldater}
Consider a Generalised Jackson Network where at time $0$ finitely many customers are already located at some servers, where at each {{of those servers}} the first customer is taking service and all others are in the queue. Assume that there is no arrival input (no new arrivals). Let $y>0$.  Suppose that the residual service time of each customer in service is at most $y$ and that the sum of the queue lengths at all servers (including customers in service) is at most $y$, too. Let $B_{md}$ denote the time to empty the network (recall that this is the {\it maximal dater}, in the terminology of \cite{BaccFoss2004}). Then there exists an absolute constant $c_{md}>0$ such that  
\[\lim_{y\to\infty}\mathbb P(B_{md}\leq c_{md}y)=1.\]
\end{lemma}
\begin{proof}
We need the following preliminary observation.
We have assumed earlier that, for a customer that arrives to the network, the probability $\beta_k := {\mathbb P} (N_k>0)$ of visiting station $k$ is positive, for all $k$. If instead a customer is already at station $k$ and its service has not started yet, then its total service time by all servers before its departure from the network, say $S^{total}_k$ is not bigger than the total service time of a customer that arrives to the network, conditioned on the event that it eventually visits station $k$. Therefore, 
${\mathbb E} S_k^{total} \le b/\beta_k\le b/\beta$ where $b = {\mathbb E}\widetilde{S}_1$ and $\beta = \min_{k \in [K]} \beta_k >0$.

Our proof relies on the monotonicity property MP1 and includes two steps. 

Step 1. Introduce service delays as follow:
At time $0$ block services at all stations but station $1$ and allow the customer in service at station $1$ (if any) to complete its service. After the service completion, the served customer either moves to another station or leaves the network.  Then we block services at all stations but station $2$ and allow the customer in service at station $2$ (if any) to complete its service. Then we repeat the same blocking/service  consequently with stations $3$ to $K$, to allow all customers that had unfinished service at time $0$ to complete it. In total, this step will take time, say, $V_1$ that is the sum of residual service times viewed at time $0$ and is at most $Ky$, by our assumptions.

Step 2: At the time $V_1$, there are no customers in service and there are, say, $y_k\equiv y_k(y)$ customers in queue $k$ (with the customer in front  awaiting service), where 
$\sum_1^K y_k \le y$ (since there is no new arrivals).  
Then the time to empty the network with these initial queues, say $V_2\equiv V_2(y_1,\ldots,y_k)$, is not bigger than the sum of total service times of all these customers,
where the $y_k$ customers from queue $k$ have total service times in the network $\sigma_{k,1}^{total}, \ldots, \sigma_{k,y_k}^{total}$, for $k=1,\ldots,K$.
All these random variables are mutually independent and, for each $k$, identically distributed. 

Assume that $y\to\infty$ along some subsequence and consider any subsequence such that all ratios $y_k/y$, $k \in [K]$ converge to some limits, say $v_k$,
where $\sum_1^k v_k \le 1$.
Then, for any $k$ such that $v_k>0$, we have
that 
\begin{align*}
\sum_1^{y_k} \sigma_{k,i}^{total}/y \to v_k {\mathbb E} \sigma_{k,1}^{total} \leq v_kb/\beta, \ \ \text{a.s.}
\end{align*}
and the limit is zero if $v_k=0$.

Then, for any $\varepsilon >0$, 
\begin{align*}
{\mathbb P} ((V_1+V_2)/y \le (1+\varepsilon)c_{md}) \to 1, \ \text{as} \ y\to\infty
\end{align*}
where $c_{md}= K+b/\beta$. 
In addition, the convergence holds 
with uniformity in all $y_k\equiv y_k(y), k=1,\ldots,K$ such that
$\sum_{k=1}^K y_k \le y$. This completes the proof of the lemma. 
\end{proof}

\begin{lemma}\label{lem:laststep}
In the conditions of \lem{maximaldater}, assume instead that there is a renewal arrival input of exogenous customers and that the stability \eq{stab}, the unboundedness \eq{unbound} conditions of \pro{Prop1} hold. We also assume that the residual arrival time, say $R_0$, of the first exogenous customer is at most $c_0y$, for some $c_0\ge 0$. Let $B_e$ be the first time when the network becomes empty. Then there exists $c>0$ such that  
\[\lim_{y\to\infty}\mathbb P(B_e\leq cy)=1.\]
\end{lemma}
\begin{proof}
Recall that blocking arriving customers at the entrance increases $B_e$, thanks to the monotonicity property MP2. 

Step 1: We block the newly arrived customers at the entrance until the first arrival after time $B_{md}$, denote this time by $T^{first}$. If $R_0 \ge B_{md}$, then the delay is $T^{first}=R(0) \le cy$. If $R(0) < B_{md}$, then, thanks to the basic facts from the renewal theory (see e.g.\ \cite[Theorem 4.3]{asmussen2003applied}), the distribution of the overshoot 
$T^{first} - B_{md}$ is stochastically bounded  by the integrated tail distribution of a typical inter-arrival time, multiplied by a constant. Therefore, ${\mathbb P}(T^{first} - B_{md}\le Cy)\to 1$ as $y\to\infty$, for any $C\ge c_0$.

Step 2: At the time $T^{first}$, there is no customers in service and there is, say, $M$ new customers waiting at the entrance. By the Law of Large Numbers, 
for any $\varepsilon >0$ and as $y\to\infty$, 
\begin{align*}
{\mathbb P} (M\le (1+\varepsilon)c_{md}y/a) \to 1 
\ \ \text{a.s.}
\end{align*}

Now we use the saturated group-$L$ network (see Section \ref{sec:auxiliary}) with at most $y$ customers waiting at the entrance as an upper bound, where $L$ is chosen sufficiently large for \eq{defL1} to hold. As we know, the saturated 
group-$L$ network may be viewed as a stable single-server queue where the large queue length decays asymptotically  linearly with rate 
$L/{\mathbb E}X^0_{1,L} - 1/a$, and, for any $\varepsilon >0$,  the time to empty this queue, say
$T_1^{empty}$,
satisfies
\begin{align*}
{\mathbb  P} (T_1^{empty} < y(1+\varepsilon)(L/{\mathbb E}X^0_{1,L} - 1/a)) \to 1, \ \text{as} \ y\to\infty.
\end{align*}
However, as we know from Section \ref{sec:ulbound}, it may take up to $J$ further busy periods of the upper-bound queue (with a total duration, say $T^J$) to empty the network. Fortunately, the distribution of $T^J$ does not depend on $y$ and
${\mathbb P} (T^J \leq h(y)) \to 1$, for any function $h(y)\to\infty$ such that $h(y)=o(y)$. 

To summarise, the total time to empty the network $B_e$ is at most $T^{first}+T^{empty}+T^J$, and, for any $\varepsilon >0$ and as $y\to\infty$, 
\begin{align*}
{\mathbb P} (T^{first}+T^{empty}+T^J \leq cy) \to 1,
\end{align*}
where
\begin{align*}
c= c_0+(1+\varepsilon)c_{md}\left(1/a+(aL/{\mathbb E}X^0_{1,L}-1)^{-1}\right).
\end{align*}
\end{proof}

\section{Fluid limits in Generalised Jackson Networks}
\label{app:fluid-limit}
\setnewcounter

In this section, we recall briefly the fluid approximation approach and provide an overview on structural properties of fluid limits in the case of Generalised Jackson Networks (GJNs). Then we consider  a particular type of initial values of stochastic  counting processes that we need.

\subsection{Overview on structural properties of fluid limits in GJNs}
\label{app:overview}

In \sectn{main}, we introduced a time-homogeneous Markov chain ~$Z_n = ((Q_{n,k}, R_{n,k})$, $k=1,2,\ldots K), n=1,2,\ldots$ and formulated the stability result (Proposition 2.1). The most modern way for proving this result is due to Dai (see Section 5 in \cite{dai1995positive}), it is based on the fluid approximation approach. This fluid approximation is also useful to study the tail probability asymptotic of the busy period of the GJN. Here we recall some basic notions and structural properties of fluid limits of GJN's that are taken or directly deduced from   \cite{dai1995positive}. Some of the related results can be found in \cite{Lela2005a,BaccFossLela2005,ChenYao2001}.

In this section, we use Lemmas 4.1--4.4 and Theorem 4.1 from \cite{dai1995positive} for analysing  fluid limits. It should be noted that Dai \cite{dai1995positive} assumes the so-called spread-out condition (see (1.5) in  \cite{dai1995positive}) in addition to \eq{stab} and \eq{unbound} for the GJN, however this extra condition is not needed for most  statements in  \cite{dai1995positive}, whilst it is crucial for Theorem 4.2 and Lemma 4.5 of that paper, which we do not use. To conclude, we only use those results of \cite{dai1995positive} that hold  without the spread-out condition.

To be consistent with the above papers, we consider instead a continuous-time Markov process \begin{align}\label{eq:zt}Z(t) = \Big(R_0(t),(Q_k(t), R_k(t)), k \in [K]\Big), t\ge 0\end{align}
where $Q_k(t)$ is the queue length at station $k$ at time $t$ (including a customer in service, if any), $R_k(t)$, $k \in [K]$ is the residual service time of the customer in service (which is assumed to be zero if the queue is empty), and $R_0(t)$ is the residual inter-arrival time , i.e. the time to the next arrival of an exogeneous customer after time $t$. We assume the process to be right-continuous. Then, clearly, $Q_{n,k} = Q_k(T_n)$ and $R_{n,k}=R_k(T_n)$, for all $n$.

Vectors $Z(t)$ take values in the space ${\cal R}_+ \times ({\cal Z}_+ \times {\cal R}_+)^K$ and we equip it with the ${\cal L}_1$ norm
\begin{align*}
|Z(t)| = R_0(t) + \sum_1^K (Q_k(t)+R_k(t)), \ \ t\ge 0.
\end{align*} 
Denote by $A_{0,k}(t)$ the counting process of the exogenous arrivals at station $k$ for $k \in [K]$, and by $S_{k}(t)$ the delayed renewal process with the sequence of the inter-counting times $\{\sigma_{k,i}; i \ge 1\}$ and with initial delay $R_{k}(0)$, where the infrastructure $\Sigma_{S}$ is used. Let $J_{k}(t)$ be the total busy time of server $k$ up to time $t \ge 0$, and let
\begin{align}
\label{eq:Dk-t}
  D_{k}(t) = S_{k}(J_{k}(t)), \qquad t \ge 0, k \in [K],
\end{align}
then we can see that $D_{k}(t)$ is the counting process of departures from station $k \in [K]$.

Let $D_{j,k}(t) = \sum_{i=1}^{D_{j}(t)} \alpha_{k,i}$ for $j \in [K]$ and $k \in [K+1]$, then $D_{j,k}(t)$ is the counting process of departures from stations $j$ which are routed to station $k$  for $j \in [K]$ or leave the network for $k=K+1$. We obviously have
$D_{k}(t) = \sum_{j=1}^{K+1} D_{k,j}(t)$. Let
\begin{align*}
  A_{k}(t) = A_{0,k}(t) + \sum_{j=1}^{K} D_{j,k}(t),
\end{align*}
then $A_{k}(t)$ is the counting processes of arrivals and service completions at station $k \in [K]$. Hence, we have 
\begin{align}
\label{eq:Ckt-2}
  Q_{k}(t) = Q_{k}(0) + A_{k}(t) - D_{k}(t), \qquad t \ge 0, k \in [K].
\end{align}
Thus, the queue length vector process $Q(\cdot)=(Q_1(\cdot),\ldots, Q_K(\cdot))$ is determined by its initial value $Q(0)$ and by
\begin{align*}
  \{A_{0,k}(t), D_{k,j}(t); k \in [K], j \in [K+1], t \ge 0\},
\end{align*}
where $[K+1] = \{1,2,\ldots, K+1\}$. Further, $R_0(0)= \inf \{t\ge 0: \sum_{k \in [K]} A_{0,k}(t)\ge 1 \}$.

We consider various initial values $Z(0) \equiv z = 
(R_0(0), ((Q_k(0),R_k(0)), k=1,\ldots,K)$ and write, for clarity,  
$Z^z(t), A_k^z(t), D_k^z(t), Q_k^z(t), D_{k,j}^z(t)$ instead of $Z(t), A_k(t)$ etc. 
Consider a family of processes indexed by $z$, that are scaled linearly in time and in space: for $|z|>0$, $t\ge 0$, 
\begin{align}
\label{eq:seq-2}
 & \widehat{Z}^z(t) \equiv (\widehat{R}_{0}(t), (\widehat{Q}_{k}(t),\widehat{R}_{k}(t)), k \in [K]) = \frac{1}{|z|} Z^z(|z|t),
\end{align}
and, for $k \in [K]$, $j \in [K+1]$, 
\begin{align}
\label{eq:seq-3}
 & (\widehat{A}_k^z(t),
\widehat{D}_k^z(t), 
\widehat{D}_{k,j}^z(t)))
= \frac{1}{|z|}
\left(
(A^z_k(|z|t),
D^z_k(|z|t),  D^z_{k,j}(|z|t))
\right).
\end{align}
Let $\overline{z}:= (\overline{R}_0(0), (\overline{Q}_k(0), \overline{R}_k(0)), k=1,\ldots,K)$ be any vector with non-negative coordinates that sum up to 1.

Here is the basic {\bf known} result we will refer to:

\begin{proposition}[Theorem 4.1 of \cite{dai1995positive}]
\label{pro:E1}
As $|z|\to\infty$ and as 
\begin{align}
\label{eq:fconv}
R_0^z(0)/|z| \to \overline{R}_0(0), R_k^z(0)/|z|\to
\overline{R}_k(0), Q_k^z(0)/|z|\to \overline{Q}_k(0), \ \text{for} \ k=1,\ldots,K,
\end{align} 
the trajectories of process $\widehat{Z}^{z}(t)$ of \eq{seq-2} converges uniformly on compact sets, $u.o.c.$ for short, to the trajectories of deterministic processes  
\begin{align}
\label{eq:flimit-1}
\overline{Z}(t) =(\overline{R}_0(t), (\overline{Q}_k(t), \overline{R}_k(t)), k \in [K]), \ \ t\ge 0\
\end{align} 
with $\overline{Z}(0)=\overline{z}$, process $\ol{Z}(t)$ is uniquely determined  by $\overline{z}$.
 \end{proposition}
 
\begin{remark}
\label{rem:E1}
In Theorem 4.1 of \cite{dai1995positive}, the uniqueness of $\ol{Z}(t)$ is not considered, but it comes from the following facts. First, it is easy to see that $\ol{Q}(t)$ is the solution of the continuous dynamic complementarity problem with reflection matrix $R=(I-P)'$ (see Definition 5.1 of \cite{dai1995positive}), where $P$ is the $K \times K$ matrix whose $(j,k)$-entry is $p_{j,k}$ for $j,k \in [K]$. Secondly, it is well known that this complementarity problem has the unique solution if $R$ is an $\sr{M}$-matrix, that is, all of its diagonal entries are positive, all of its off-diagonal entries are not positive and it is invertible (e.g., see Theorem 7.2 of \cite{ChenYao2001}). Since $I-P'$ is an $\sr{M}$-matrix, $\ol{Q}(t)$ is uniquely determined.
\end{remark}

\begin{corollary}\rm
\label{cor:E1}
Under the same condition \eq{fconv} as in \pro{E1}, the process \eq{seq-3} converges $u.o.c.$ to the deterministic process 
\begin{align}
\label{eq:flimit-2}
( \overline{A}_k(t),  \overline{D}_k(t),  \overline{D}_{k,j}(t)), \ \ t\ge 0, \ k \in [K], \ j \in [K+1],
\end{align}
which is uniquely determined by $\ol{z}$. 
\end{corollary}
This corollary is immediate from Lemma 4.2 of \cite{dai1995positive}, \eq{Dk-t} and \pro{E1}. It is notable that this Lemma 4.2 was proved in \cite{dai1995positive} by the strong law of large numbers for the counting processes $A_{0,k}(t)$ and $S_{k}(t)$ for $k \in [K]$ under the station-based infrastructure $\Sigma_{S}$, and the spread out condition (1.5) from  \cite{dai1995positive} is not needed.

 Processes \eq{flimit-1} and \eq{flimit-2} have a number of nice properties that are presented below.

\begin{proposition}\label{propE2}
Let $k \in \{1,\ldots,K\}$. 
\begin{itemize}
\item{(p1)} 
\begin{itemize}
\item  If 
$\overline{R}_k(0)=0$, then  
$\overline{R}_k(t)=0$ for all $t>0$, 
\item  and 
if $\overline{R}_k(0)>0$, then $\overline{R}_k(t)$ decreases linearly with rate $-1$ until arriving to $0$ at time $t=\overline{R}_k(0)$ and stays at $0$ after that (i.e. $\overline{R}_k(u)=0$ for $u>t$). 
\item Further, 
there is no departure of fluid from station $k$ within time interval $[0, \overline{R}_k(0)]$, i.e.  
$\overline{D}_k(u)=0$ for all $u \in [0, \overline{R}_k(0)]$.
\end{itemize}

\item{(p2)}
We have  
\begin{align}\label{eq:AD}
\overline{A}_k(t)&=\overline{A}_{0, k}(t) + \sum_{j=1}^K \overline{D}_{j,k}(t),
\qquad
\overline{D}_k(t) = \sum_{j=1}^{K+1} \overline{D}_{k,j}(t),  \\ \overline{Q}_k(t) &= \overline{Q}_k(0)+
\overline{A}_k(t)-\overline{D}_k(t).
\end{align}

\item{(p3)}
The limiting processes \eq{flimit-1} and \eq{flimit-2} are continuous and piecewise
 linear, and there are finitely many time instants  where the linear functions change their directions (we may call these exceptional times as ``changing times/points''). In particular, the processes are differentiable everywhere except the changing times (recall that a point  where a process is differentiable is usually named a {\bf  regular point}).
Since the sample paths of those limiting processes are continuous and piecewise linear, they have right derivatives from the right everywhere, that coincide with derivatives at regular points. In what follows,we will use the term ``derivative" in the sense of derivative from the right.

\item{(p4)}
At any point of time $t \ge 0$, we may introduce the following derivatives.
\begin{align*}
\overline{a}_k(t) = (\overline{A}_k(t))',  \ 
\overline{d}_k(t) = (\overline{D}_k(t))', \ 
\overline{q}_k(t) = (\overline{Q}_k(t))', \
\overline{d}_{j,k}(t) = (\overline{D}_{j,k}(t))'.
\end{align*} 

Hence, by \eq{AD} and thanks to the SLLN,
\begin{align*}
\overline{d}_{j,k}(t) = \overline{d}_j(t) \cdot p_{j,k},  \ \text{and then} \
\overline{a}_k(t)=\overline{a}_{0, k}(t) + \sum_{j=1}^K \overline{d}_{j}(t) p_{j,k}. 
\end{align*}
Further, 
\begin{align*}
\overline{q}_k(t) = \overline{a}_k(t) - \overline{d}_k(t).
\end{align*}
 
\item{(p5)}
At any regular point $t>0$, if 
$\overline{R}_k(t)=0$ and $\overline{Q}_k(t)>0$, then
the fluid departs from station $k$ with the maximal rate, i.e. 
$\overline{d}_k(t)=1/b_k$ and then
 the derivatives $\overline{d}_{k,j}(t):= 
 (\overline{D}_{k,j}(t))' =p_{k,j}/b_k$. 

\item{(p6)}
At any regular point $t>0$, if $\overline{R}_k(t)=0$ and $\overline{Q}_k(t)=0$, then 
$\overline{q}_k(t)=0$, too. This means that the arrival and the departure rates coincide, 
$\overline{a}_k(t) =
\overline{d}_k(t)\le 1/b_k$. 
 
\item{(p7)}
Under the stability condition \eq{stab}, if $\overline{Q}_k(t)=0$ for some $k$ at a regular point $t>0$ and if $\max_{0\le k \le K} \overline{R}_k(0)\le t$, then $\overline{Q}_k(u)=0$, for all $u>t$. 

\item{(p8)}
Under the stability condition \eq{stab}, the time to empty the fluid network
\begin{align*}
\tau_{\overline{z}} = \inf \{t\ge 0 :  \ 
|\overline{Z}(t)|=0 \}
\end{align*}
is finite for any $\overline{z}$ and, moreover,
$\sup \tau_{\overline{z}}$ is finite, too, where the supremum is taken over all $\overline{z}$ with
$|\overline{z}|=1$.
\end{itemize}
\end{proposition}
 
\subsection{Particular case}\label{app:pc}

We are interested in the particular initial conditions for the Markov process $Z(t)$, that are: for a fixed $k \in [K]$, 
\begin{align}\label{eq:ini}
R_k(0)=x; \ \ R_j(0)\le h(x) \  \text{for} \  j\neq k; \ \ Q_k(0)\le h(x) \ \text{for all}  \ k.
\end{align}
Here $h(x)$ is again any nonnegative monotone function that increases to infinity slower than the linear function,  $h(x)=o(x)$.
Conditions \eq{ini} imply that $x \le |z|\le x+2Kh(x)$.
Further, for a fixed $k=1,\ldots,K$ and as $|z|\to\infty$, all these processes converge to the same fluid limit with
initial values
\begin{align}\label{eq:our-ini}
\overline{R}_k(0)=1; \ \ \overline{R}_j(0)= 0 \  \text{for} \  j\neq k; \ \ \overline{Q}_k(0)=0 \ \text{for all}  \ k.
\end{align}
Denote this fluid limit as 
\begin{align}\label{eq:zk}
\overline{Z}^{(k)}(t)= (\overline{R}_0^{(k)}(t), \ 
(\overline{R}^{(k)}_j(t), \overline{Q}^{(k)}_j(t)),\ j=1,\ldots,K)
\end{align}
and also supply with the upper $(k)$ in all other characteristics. 
\begin{proposition}
\label{propE4}
As it follows from the previous subsection, under the stability conditions \eq{stab}, the process \eq{zk} with initial values \eq{our-ini} has the following properties:
at any regular point $t\in [0,1)$,  
\begin{itemize}
\item[(p9)]
\begin{itemize}
\item
$\overline{R}^{(k)}_k(t)$ decreases linearly with rate $-1$; 
\item
$\overline{R}^{(k)}_0(t)=0$ and 
$(\overline{R}^{(k)}_j(t),\overline{Q}^{(k)}_j(t))=(0,0)$, for all $j=1,\ldots,K$, $j\neq k$.
\end{itemize}
\item[(p10)]
\begin{itemize}
\item
$\overline{d}_k^{(k)}(t)=0$; 
\item
for any 
$j\neq k$, the    
arrival $\overline{a}_j^{(k)}:=\overline{a}_j^{(k)}(t)$ and the departure 
$\overline{d}_j^{(k)}:=\overline{d}_j^{(k)}(t)$ rates of fluid to/from station $j$  are constants that must coincide, $\overline{a}_j^{(k)} =  \overline{d}_j^{(k)}$, and  the following equations hold:
\begin{align}\label{eq:a=d}
p_{0,j}/a + \sum_{\ell=1}^K p_{\ell,j}\overline{d}_{\ell}^{(k)} =
\overline{a}_j^{(k)} = \overline{d}_j^{(k)}, 
\end{align}
\item
Further, 
$\overline{Q}^{(k)}_k(t)$ increases linearly with rate  
\begin{align}\label{eq:ratek}
\overline{a}_k^{(k)} = p_{0,k}/a + \sum_{1\le j \le K, j\neq k}
p_{j,k}\overline{d}^{(k)}_j=\frac{\beta_k}{a}
\end{align}
where, recall,
$\beta_k={\mathbb P} (N_k>0)>0$.  Indeed, the second equality above is true since $\beta_k$ is the probability that an arriving customer ever visits station $k$ and since, at the fluid level, any arriving customer immediately either arrives to station $k$ or leaves the network. Thus, the fluid level at server $k$ at time $1$ is $\beta_k/a.$
\end{itemize}
\end{itemize}
Thus, at time $t=1$ all coordinates of the vector 
$\overline{Z}^{(k)}(1)$ are equal to $0$ except 
$\overline{Q}^{(k)}_k(1)=\beta_k/a$.
\end{proposition}

Under the stability condition \eq{stab}, it is well known that $\ol{Q}(t)$ vanishes in a finite time (e.g., see the proof of Theorem 5.1 of \cite{dai1995positive}). Furthermore, it is not hard to see that, by the monotonicity properties, the
departure rate from queue $k$ cannot increase when it empties. Based on these observations, we have the following proposition.

\begin{proposition}\label{propE3}
Consider again the fluid process \eq{zk} with initial values \eq{our-ini} and 
assume the stability condition \eq{stab} to hold.  
\begin{itemize}
\item {(p11)} 
Based on properties (p1) and (p7),
there exist exactly $K$ ``changing"  times  $\theta_{Q,j}\ge 1$, $j=1,\ldots,K$ such that    $\overline{Q}_j(t)>0$ for all $t \in (1, \theta_{Q,j})$ and $\overline{Q}_j(t) = 0$ for all $t \ge \theta_{Q,j}$. 
\item{(p12)}
Note that some of the $\theta_{Q,j}$'s may coincide. Assume that there are $K'+1$ distinct numbers among them, and order them as follows
\begin{align*}
\theta^{(0)}\equiv 1 < \theta^{(1)} < \theta^{(2)} < \ldots < 
\theta^{(K')}.
\end{align*}
We may  conclude that, for $m=0,1,\ldots, K'$,
within any open time interval $(\theta^{(m)},\theta^{(m+1)})$, the number $M_m := |\mathcal{A}^{(m)}|$, where $\mathcal{A}^{(m)} = \{j: \ \overline{Q}_j(t)=0, \forall t\in (\theta^{(m)},\theta^{(m+1)})\}$,
of coordinates with zero fluid level does not change and, therefore, the arrival and the departure rates at any server are constants.  Further, the numbers $M_m$ are
increasing in $m$. 
Then, it follows from (p4), (p5) and (p6) that the arrival  and the departure rates at any station are non-increasing  in $m=0,\ldots,K'$.
\item{(p13)} 	For any time interval $(\theta^{(m)}, \theta^{(m+1)})$, the arrival and the departure rates at any server $j$ (denote them $\overline{a}_j^{(k,m)}$ and $\overline{d}_j^{(k,m)}$) are {uniquely determined (see Remark \ref{rem:E1}) by
\begin{itemize} 
\item the uniqueness of 
the set  ${\mathcal A}^{(m)}$  
\item and the uniqueness of solutions in  the following system of linear equations:
\begin{align}\label{eq:spirit}\begin{split}
\text{for} \ j\in \mathcal{{A}}^{(m)}, \ \ 
&p_{0,j}/a + \sum_{\ell=1}^K p_{\ell,j}\overline{d}_{\ell}^{(k,m)} =
\overline{a}_j^{(k,m)} = \overline{d}_j^{(k,m)}, \\
\text{for} \ j\in [K]\setminus\mathcal{{A}}^{(m)}, \ \ 
&
p_{0,j}/a + \sum_{\ell=1}^K p_{\ell,j}\overline{d}_{\ell}^{(k,m)} =
\overline{a}_j^{(k,m)}  \ \text{and} \ \ 
\overline{d}_j^{(k,m)} = 1/b_j. 
\end{split}\end{align}
\end{itemize}}
 \end{itemize}
 \end{proposition}
\bigskip 

Finally, we show how to obtain $\tau^{(k)},$ the time needed to empty the network. For $t\ge 0$, denote by 
\begin{align*}
\overline{Z}^{(k,1)}(t) := \frac{a}{\beta_k} \overline{Z}^{(k)}(t+1)
\end{align*}
the auxiliary fluid limit with initial values
\begin{align}\label{eq:our-ini2}
\overline{Q}_k(0)=1; \ \ \overline{Q}_j(0)= 0 \  \text{for} \ j=1,\ldots, K, \  j\neq k; \ \ \overline{R}_k(0)=0 \ \text{for}  \ k=0,1,\ldots,K.
\end{align}
Assume the stability assumption \eq{stab} to hold and let $\tau^{(k,1)}$ be the time to empty this fluid limit,
\begin{align*}
\tau^{(k,1)} = \inf \{t>0: \ |\overline{Z}^{(k,1)}(t)|=0\},
\end{align*}
Then the standard scaling argument implies that the time to empty the fluid limit $\tau^{(k)}$
is nothing else than
\begin{align}\label{eq:tauk}
\tau^{(k)}= 1+ \frac{\beta_k}{a} \cdot \tau^{(k,1)}.
\end{align}

To find $\tau^{(k,1)}$, we may use the property 
(p13) above.


\end{document}